\documentclass[english,11pt]{article}
\usepackage[english]{babel}
\usepackage{textcomp}
\usepackage{eurosym}
\usepackage{amsthm}

\usepackage[intlimits]{amsmath}
\usepackage[all]{xy}
\usepackage{paralist}
\usepackage{amssymb}
\usepackage{amsfonts}
\usepackage{mathbbol}
\usepackage{paralist}
\usepackage{graphicx}
\usepackage{fancyhdr}
\usepackage{epsfig}
\usepackage[linkcolor=black,citecolor=black]{hyperref}
\usepackage[latin1]{inputenc}
\usepackage[T1]{fontenc}
\usepackage{ifthen}
\usepackage{amsfonts}
\usepackage{amsmath}
\usepackage{amssymb}
\usepackage{graphicx}
\usepackage[]{color}
\usepackage{upref}
\usepackage[all]{xy}
\usepackage{paralist}
\usepackage{textcomp}

\theoremstyle{plain}

\newtheorem{thms}{Theorem}[section]

\theoremstyle{definition}

\newtheorem{rems}[thms]{Remark}

\numberwithin{equation}{section}

\theoremstyle{definition}

\numberwithin{equation}{section}

\newcommand{\DD}{\mathbb{D}}
\newcommand{\EE}{\mathbb{E}}

\newcommand{\NN}{\mathbb{N}}

\newcommand{\PP}{\mathbb{P}}

\newcommand{\RR}{\mathbb{R}}

\newcommand{\bB}{\mathcal{B}}

\newcommand{\fF}{\mathcal{F}}

\newcommand{\iI}{\mathcal{I}}

\newcommand{\vt}{\vartheta}
\newcommand{\al}{\alpha}

\newcommand{\e}{\varepsilon}
\newcommand{\la}{\lambda}

\newcommand{\si}{\sigma}

\newcommand{\om}{\omega}
\newcommand{\Om}{\Omega}

\newcommand{\ra}{\rightarrow}

\newcommand{\lra}{\longrightarrow}
\newcommand{\ti}{\widetilde}
\newcommand{\ho}{\hat}

\newcommand{\vzv}{\Leftrightarrow}

\newcommand{\ind}{\mathbf{1}}

\newcommand{\lqq}{\leqslant}
\newcommand{\gqq}{\geqslant}

\usepackage[]{color}
\definecolor{DarkGreen}{rgb}{0.1,0.7,0.3}   

\DeclareMathOperator{\Li}{\ensuremath{Li}}

\providecommand{\U}[1]{\protect\rule{.1in}{.1in}}
\newtheorem{theorem}{Theorem}

\newtheorem{corollary}[theorem]{Corollary}

\newtheorem{definition}[theorem]{Definition}

\newtheorem{lemma}[theorem]{Lemma}

\newtheorem{proposition}[theorem]{Proposition}
\newtheorem{remark}[theorem]{Remark}

\definecolor{DarkGreen}{rgb}{0,0.7,0.2}   
\definecolor{DarkBlue}{rgb}{0,0,0.7}
\definecolor{DarkRed}{rgb}{0.95,0,0}

\begin{document}
\title{A solution selection problem with small symmetric stable perturbations}

\author{Franco Flandoli, \footnote{Dipartimento di Matematica, Largo Bruno
Pontecorvo 5, 56127 Pisa, Italy; flandoli@dma.unipi.it}
\hspace{1.5cm}
Michael H\"ogele\footnote{Institut f\"ur Mathematik, Universit\"at Potsdam, Germany; hoegele@math.uni-potsdam.de}
}

\maketitle

\begin{abstract}
The zero-noise limit of differential equations with singular coefficients 
is investigated for the first time in the case when the noise is an $\alpha $-stable process. 
It is proved that extremal solutions are selected and the respective probability of selection is computed. 
For this purpose an exit time problem from the half-line, which is of interest in its own right, 
is formulated and studied by means of a suitable decomposition in small and large jumps adapted to the singular drift.
\end{abstract}




\section{Introduction}

The zero-noise limit of a stochastic differential equation, with drift vector field $b$ and a Wiener process $W$, say of the form%
\begin{equation}
X_{t}^{\varepsilon}=x_{0}+\int_{0}^{t}b\left(  X_{s}^{\varepsilon}\right)
ds+\varepsilon W_{t} , \quad t\gqq 0, \e>0, \label{SDE1}%
\end{equation}
is a classical subject of probability, see for instance \cite{FreiWen}. When
the limit deterministic equation%
\begin{equation}
X_{t}=x_{0}+\int_{0}^{t}b\left(  X_{s}\right)  ds \label{ODE1}, \quad t\gqq 0, 
\end{equation}
is well posed, usually one has $X_{t}^{\varepsilon}\rightarrow X_{t}$ a.s. and
typical relevant questions are the speed of convergence and large deviations.
On the contrary, when the Cauchy problem (\ref{ODE1}) has more than one
solution, the first question concerns the \textit{selection}, namely which
solutions of (\ref{ODE1}) are selected in the limit and with which
probability. This selection problem is still poorly understood and we aim to
contribute with the investigation of the case when the noise is an $\alpha
$-stable process.

The case treated until now in the literature is the noise of Wiener type. All
known quantitative results are restricted to equations in dimension one. The
breakthrough on the subject was due to Bafico and Baldi \cite{BB} who solved the
selection problem for very general drift $b$ having one point $x_{0}$ of
singularity. The paradigmatic example of $b$ to test the theory is%
\begin{equation}
b\left(  x\right)  =\left\{
\begin{array}
[c]{ccc}%
B^{+}\left\vert x\right\vert ^{\beta^{+}} & \text{for} & x\geq0\\
-B^{-}\left\vert x\right\vert ^{\beta^{-}} & \text{for} & x<0.
\end{array}
\right.  \label{drift b}%
\end{equation}
where $B^{\pm}>0$, $\beta^{\pm}\in\left(  0,1\right)$; the deterministic
equation (\ref{ODE1}) with $x_{0}=0$ has infinitely many solutions, which are
equal to zero on $[0,\infty)$ or on some interval $[0,t_{0}]$ (possibly
$t_{0}=0$) and then, on $[t_{0},\infty)$, they are equal either to
$C^{+}\left(  t-t_{0}\right)  ^{\frac{1}{1-\beta^{+}}}$ or to $-C^{-}\left(
t-t_{0}\right)^{\frac{1}{1-\beta^{-}}}$, with $C^{\pm}$ given in (\ref{eq: explicit solution}) 
of Section \ref{sec: large jumps}; 
a central role will be played by the two \textit{extremal}
solution,
\[
x^{\pm}=\pm C^{\pm}t^{\frac{1}{1-\beta^{\pm}}}.
\]
The article \cite{BB} completely solves the selection problem for this and more
general examples, making use of explicit computations on the differential
equations satisfied by suitable exit time probabilities; such equations are
elliptic PDEs, in general, so they are explicitly solvable only in dimension
one (except for particular cases). The final result is that the law
$P_{\varepsilon}^{W}$, on $C\left(  \left[  0,T\right]  ;\mathbb{R}\right)  $,
of the unique solution $X_{t}^{\varepsilon}$ of equation (\ref{SDE1}) with
$x_{0}=0$ and $b$ as in (\ref{drift b}), satisfies%
\[
P_{\varepsilon}^{W}\overset{w}{\longrightarrow}p^{+}\delta_{x^{+}}+p^{-}%
\delta_{x^{-}},%
\]
where $p^{-}=1-p^{+}$ and 
\begin{equation}\label{eq: symmetric exit probabilities}
p^{+}=\left\{
\begin{array}
[c]{ccc}%
1 & \text{if} & \beta^{+}<\beta^{-}\\
\frac{(B^{-})^{-\frac{1}{1+\beta}}}{(B^{+})^{-\frac{1}{1+\beta}}%
+(B^{-})^{-\frac{1}{1+\beta}}} & \text{if} & \beta^{+}=\beta^{-}=:\beta\\
0 & \text{if} & \beta^{+}>\beta^{-}.
\end{array}
\right.
\end{equation}
This or part of this result was re-proved later on using other approaches, not
based on elliptic PDEs but only on tools of stochastic analysis and dynamical
arguments, see \cite{DFV, Trev}. These investigations are also
motivated by the fact that in dimension greater than one the elliptic PDE
approach is not possible.

The aim of this paper is to investigate these questions when the Wiener
process $W$ is replaced by a general pure-jump $\alpha$-stable process $L$. 
This process satisfies for any $a>0$ the following self-similarity condition 
$(L_{a t})_{t\gqq 0} \stackrel{d}{=} (a^\frac{1}{\al} L_t + \gamma_0 t)_{t\gqq 0}$, 
for a drift $\gamma_0\in \RR$ which accounts for the asymmetry of the law of $L$. 
The stochastic differential equation, then, takes the form%
\begin{equation}
X_{t}^{\varepsilon}=x_{0}+\int_{0}^{t}b\left(  X_{s}^{\varepsilon}\right)
ds+\varepsilon L_{t}, \qquad t\gqq 0, \e>0.\label{SDE}%
\end{equation}
Here explicit solution of the elliptic equations for exit time probabilities
are not feasible and thus it is again an example where we need to understand
the problem with new tools and ideas. This feature is similar to the theory of
asymptotic first exit times for equations with regular coefficients and small
noise, see \cite{DHI13, HoePav-14, ImkellerP-06, Pavlyukevich11} 
for recent progresses in the case of L\'{e}vy noise.
This requires a careful understanding of the role of small and large jumps,
which is conceptually new and interesting; technically the more demanding part
is the estimate of the Laplace transform of the exit times. Some ingredients
are also inspired by \cite{DFV}.

The main result is the following theorem.

\begin{theorem}\label{main theorem}
If $\alpha>1-\left(  \beta^{+}\wedge\beta^{-}\right)$ and $\beta^+ \neq \beta^-$, 
then for any $T>0$ 
\[
P_{\varepsilon}^{L}\overset{w}{\longrightarrow}p^{+}\delta_{x^{+}}+p^{-}%
\delta_{x^{-}}%
\]
where $P_{\varepsilon}^{L}$ is the
law, on Skorohod space $\mathbb{D}\left(  \left[  0,T\right]  ;\mathbb{R}%
\right) $, of the unique solution $X_{t}^{\varepsilon}$ of equation
(\ref{SDE}) with $x_{0}=0$ and $p^{+}, p^- = 1- p^+$ are given by (\ref{eq: symmetric exit probabilities}).
\end{theorem}

The time interval where this convergence takes place can be chosen to be any
bounded interval $\left[  0,T\right]  $, but with a suitable reformulation of
the result it may also be an interval which increases like $\left[
0,\varepsilon^{-\theta^*}\right]  $, for suitable $\theta^*>0$, see the technical
statements below; this is a novelty compared with the literature on the
Brownian case. 

The condition $\al > 1- (\beta^+ \wedge \beta^-)$ 
appears naturally in the investigation of the local behavior of $X^\e$ close to 
the origin. It states that there exists a time scale $t_\e$ below which, 
the solution behaves mainly ``noise-like'', while for scales larger than $t_\e$ 
the drift takes over irresistibly. The condition ensures that this critical time scale 
tends to $0$. 

For this purpose we study an asymptotic first exit problem for the strong solution $X^\e$ 
of (\ref{SDE}) from a half-line. This is a problem in its own right. 
The proof of this result yields an asymptotic lower bound of $X^\e$ for times 
beyond the occurrence of the first ``large'' jump in an appropriate sense 
as stated in Corollary \ref{cor: short time scale convergence}. 
Before such first large jump, that is on a time scale up to $\e^{-\theta^*}$ 
however, the system exhibits the mentioned behavior similar to a Brownian perturbation. 
Among the other technical novelties, there is the use of the
linearized system in order to show that excursions away from the origin are
large enough.

It is well-known in the literature \cite{DHI13, HoePav-14b, IP08} 
in the case of systems of stable fixed points or attractors 
perturbed by a stable perturbation $\e L$, that the critical time scale is given by $\e^{-\al}$. 
The following exit time problem establishes 
that the critical time scale is larger than $\e^{-\al}$. 

\begin{theorem}\label{thm: exit}
For any $\beta^+ \in(0,1)$ and $\al \in (0,2)$ 
there is a monotonically increasing, continuous functions $\delta^+_\cdot: (0,1)\ra (0,1)$ 
of polynomial order with $\delta_\e\ra 0$ as $\e\ra0$ such that the first exit time 
\[
\tau^{x,\e, -} := \inf\{t>0~|~X^{\e, x}_t \lqq \delta_\e^+\} 
\]
of the solution $X^{x,\varepsilon}$ of (\ref{SDE}) 
satisfies for any function $m_\e \ra \infty$ with $\limsup_{\e \ra 0} m_\e \e^\al <\infty$  
\[
\lim_{\e\ra 0} \sup_{x\gqq 3\delta_\e^+} \PP(\tau^{x,\e, -} \lqq m_\e) = 0. 
\]
\end{theorem}

The article is structured as follows. After a brief set of notations, 
we show the previously first exit result of Theorem \ref{thm: exit} 
in Section~\ref{sec: large jumps}. 
This is carried out for initial values which may approach $0$ as a function of $\e$, 
however only sufficiently slowly, as $\e\ra 0$. 
Section~\ref{sec: close to the origin} zooms into the behavior of the solution 
in a space-time box of short temporal and spatial scales around the origin and determines 
the exit probabilities to each spatial side of the box 
with the help of the self-similarity of the driving L\'evy noise. 
In Section~\ref{sec: linearized} it is shown that an unstable 
linearized intermediate regime stabilizes the exit direction from 
the small environment of the origin and rapidly enhances the solution 
until it reaches the area of initial values for the regime in Section~\ref{sec: large jumps}. 
In Section~\ref{sec: Hauptsatzbeweis} we prove a slightly stronger result, which implies Theorem \ref{main theorem}.

\section{Preliminaries}\label{sec: preliminaries}

For the following notation we refer to Sato \cite{Sato-99}.  
A L\'evy process $L$ with values in the real line 
over a given probability space $(\Omega, \fF, \PP)$ is 
a stochastic process $L = (L_t)_{t\gqq 0}$ starting in $0\in \RR$ 
with independent and identically distributed increments. 

The L\'evy-Khintchine formula establishes the following representation of the characteristic function of the marginal law of the 
L\'evy process $Z$. There exists a drift $\gamma \in \RR$, $\si>0$ and 
a $\sigma$-finite Borel measure $\nu$ on $\RR$, 
the so-called \textit{L\'evy measure}, satisfying 
\begin{equation}\label{eq: def Levy measure 1}
\nu\{0\}=0, \qquad \mbox{ and }\quad \int_{\RR}(1 \wedge |u|)^2\nu(du)<\infty, 
\end{equation}
such that for any $t\gqq 0$ the characteristic function reads 
\begin{align}
&\EE[ e^{i z L_t}] = e^{t\psi(z)},  \quad z\in \RR,\nonumber\\
&\psi(z) = i \gamma z - \frac{\si^2 z^2}{2} 
+ \int_{\RR} (e^{i z y}-1 - i z y\ind\{|y|\lqq 1\}) \nu(dz).
\label{eq: Levy-Khinchin}
\end{align}
The triplet $(\gamma, \si, \nu)$ determines the process $L$ in law uniquely. 

A symmetric $\alpha$-stable process $L$ in law for $\al\in (0,2)$ is a L\'evy process 
with canonical triplet $(0, 0, \nu)$, 
where $\nu$ is given as 
\begin{equation}\label{def: nu}
\nu(dy) =  \frac{c}{y^{\al+1}} \ind\{y< 0\}+\frac{c}{y^{\al+1}} \ind\{y> 0\}, 
\end{equation}
with $c>0$.

A symmetric $\al$-stable processes $L$ in law satisfies the following self-similarity property.  
Given the L\'evy measure associated to $\nu$ for any $a>0$  
\begin{equation}
(L_{a t})_{t\gqq 0} \stackrel{d}{=} (a^\frac{1}{\al} L_t)_{t\gqq 0}.
\end{equation}
For details consult \cite{Sato-99}, Section 8 and 14. 

The L\'evy-It\^o decomposition \cite{Sato-99}, Theorem 19.2, 
yields the pathwise representation 
\begin{equation}\label{eq: Levy Ito}
L_t = \int_{0}^t \int_{0< |y|\lqq 1} y (N(ds dy)-ds \nu(dy)) + \int_0^t \int_{|y|> 1} y N(ds dy) \qquad \mbox{ for all }t\gqq 0, \PP-\mbox{a.s., }
\end{equation}
where $N([0, t] \times B, \om) = \#\{s\in [0, t] \in \RR~|~(s, \Delta L_t(\om)) \in B\}$ for $t \gqq 0$, $B\in \bB(\RR)$ and $\om \in\Om$, 
is the Poisson random measure associated to $dt \otimes \nu$. 

\begin{proposition}
Let $\beta^+, \beta^- \in (0, 1)$, $B^+, B^->0$ and $L$ be a pure jump $\alpha$-stable process with 
$\al\gqq 1-(\beta^+\wedge \beta^-)$ over a given filtered probability space given by (\ref{eq: Levy Ito}). 
Then equation (\ref{SDE}) with these coefficients has a unique strong solution, 
which satisfies the strong Markov property. 
\end{proposition}

The result is given in Tanaka \cite{TTW74}.

\section{An exit problem from the half-line: Proof of Theorem \ref{thm: exit}}\label{sec: large jumps}
\subsection{Proof of Theorem \ref{thm: exit}}

The proof of Theorem \ref{thm: exit} is structured in four parts. 
After the technical preparation and two essential observations 
we derive the main recursion. In the last part we conclude. 

\paragraph{1) Setting and notation: } Let us denote $u(t; x) := X^{x, 0}_t$ for convenience.
The first observation is the following. 
Let $\delta>0$ and $x\in \RR$ an initial value with $|x|>\delta$. 
Then $b\big|_{\mathbb{R}\setminus [-\delta, \delta]}$ 
satisfies global Lipschitz and growth conditions, such that there exists a unique 
strong local solution, which lives until to the stopping time 
\[
\tau^{x, \varepsilon, \delta} := \inf\{t>0~|~X^{x,\varepsilon}_{t} \in [-\delta, \delta]\}. 
\]
Here the Lipschitz constant depends essentially on $\delta$ and explodes as \mbox{$\delta \searrow 0$.} 
As usually in this situation, we divide the process $L = \eta^\varepsilon + \xi^\varepsilon$ by 
a $\varepsilon$-dependent threshold $\varepsilon^{-\rho}$, 
where $\rho\in (0,1)$ is a parameter to be made precise in the sequel.
More precisely the compound Poisson process with 
\[
\eta^\varepsilon_t = \sum_{i=1}^\infty W_i \ind\{T_i \lqq t\}
\]
with arrival times $T_i = \sum_{j=1}^i t_i$, where $t_i$ i.i.d. waiting times 
and i.i.d. ``large'' jump increments $(W_i)_{i\in \mathbb{N}}$ with the conditional law 
\begin{align}
W_i &\sim \frac{1}{\lambda_\varepsilon}\nu(\cdot \cap (\mathbb{R} \setminus [-\varepsilon^{-\rho}, \varepsilon^{-\rho}]))\label{eq: Wi}\\
t_i &\sim \mbox{EXP}(\la_\e)\qquad \mbox{ for } i\in \NN,\label{eq: ti}
\end{align}
where
\begin{equation}\label{eq: la eps}
\lambda_\varepsilon = \nu(\mathbb{R} \setminus [-\varepsilon^{-\rho}, \varepsilon^{-\rho}]) 
= 2 \int_{\e^{-\rho}}^\infty \frac{dy}{y^{\alpha+1}} = \frac{2}{\alpha} \e^{\alpha\rho},
\end{equation}
and the remaining semi-martingale 
\begin{align}
\xi^\varepsilon = L - \eta^\varepsilon 
\end{align}
with uniformly bounded jumps, which implies 
the existence of exponential moments.
Let us denote by $Y^{x, \varepsilon}$ the solution of 
\begin{equation}
Y_{t}^{x,\varepsilon}=x + \int_0^t b\left(  Y_{s}^{x,\varepsilon}\right)  ds+\varepsilon
\xi^\varepsilon_t, 
\label{SDE small jumps}%
\end{equation}
which exists uniquely under the same conditions as does $X^{x, \varepsilon}$. 
For $\delta>0$ we fix the notation 
\begin{align*}
D_{\delta}^+ &:= (\delta, \infty). 
\end{align*}
For a function $\delta_\cdot: (0,1) \ra (0,1)$ with $\delta_\e\searrow 0$ 
to be specified later we fix 
\begin{align*}
&\tau^{x, \varepsilon, -} := \inf\{t>0~|~X^\varepsilon_{t,x} \notin D_{\delta_\varepsilon}^+\}.
 \end{align*}

\paragraph{2) Two observations: } 
The following observations reveal the first exit mechanism. 
\paragraph{2.1) Up to the first large jump, the deterministic solutions travel sufficiently far: } 
Separation of variables yields the explicit representation for $t\gqq t'$ and $x\gqq 0$
\begin{equation}\label{eq: explicit solution}
u(t;t', x) = \left(B (1-\beta) (t-t') +  x^{1-\beta}\right)^{\frac{1}{1-\beta}}.
\end{equation}
Hence for $z\gqq x$ and $t'=0$, we obtain 
\begin{align*} 
P\left(u(T_1;x) \gqq z\right) & = P\big(\big(B (1-\beta)T_1+  x^{1-\beta}\big)^{\frac{1}{1-\beta}}\gqq z\big) \\
& = P\big(T_1 \gqq \frac{z^{1-\beta}-x^{1-\beta}}{B(1-\beta)} \big)\\
& = \exp\Big(- (z^{1-\beta}-x^{1-\beta})\frac{\lambda_\varepsilon}{B(1-\beta)} \Big) \\
& = P(Z\gqq z ~|~ Z\gqq x). 
\end{align*}
This is the tail of the distribution function of a Weibull distributed random variable $Z$ with 
shape parameter $1-\beta$ and scaling parameter 
\[\Big(\frac{\lambda_\varepsilon}{B(1-\beta)}\Big)^{\frac{1}{1-\beta}} = \frac{\e^{\frac{\alpha\rho}{1-\beta}}}{B(1-\beta)^{\frac{1}{1-\beta}}}\]
conditioned on the event $\{Z\gqq x\}$. 
We define for $\Gamma >1$ such that $\Gamma < \frac{1}{1-\beta}$ and 
\begin{align*}
\gamma_\varepsilon &:= (\la_\e^{-\frac{1}{\Gamma}} - (3\delta_\e)^{1-\beta})^{\frac{1}{1-\beta}} \approx_\e \e^{-\frac{\alpha}{\Gamma} \frac{\rho}{1-\beta}}.
\end{align*}
Hence 
\begin{align}\label{eq: det escape from 0}\
\lim_{\e\ra 0+} \sup_{x\in D_{3\delta_\e}} \PP(u(T_1;x) \gqq \gamma_\e) \ra 1,
\end{align}
and 
\begin{align*}
\sup_{3\delta_\e \lqq x\lqq \gamma_\e} \PP(u(T_1;x) \lqq 2\gamma_\e) 
&\approx_\e \la_\e^{1-\frac{1}{\Gamma}}.
\end{align*}

\paragraph{2.2) Control the deviation of the small jump solution from the deterministic solution: } 
For each $\rho\in (0,1)$ there are functions $\delta_\cdot: (0,1) \ra (0,1)$, \mbox{$r^\cdot: (0,1) \ra (0, \infty)$} such that 
\begin{align*}
\e^{\alpha \rho} r^\e \ra \infty\quad \mbox{ and } \quad \frac{\delta_\e}{\e^{1-\rho } r^\e} \ra \infty . 
\end{align*}
Put in other terms the first result means $r^\e \gtrsim_\e \frac{1}{\e^{\al\rho}}$. 
We define 
\begin{equation}\label{def: r eps}
r^\e := \frac{|\ln(\e)|^2}{\e^{\al\rho}}. 
\end{equation}
For the second expression we have  
\begin{equation}\label{eq: delta limit}
\infty \leftarrow \frac{\delta_\e}{\e^{1-\rho } r^\e} = \frac{\delta_\e}{\e^{1-\rho- \alpha\rho}} \frac{1}{\e^{\alpha\rho} r^\e}.
\end{equation}
Therefore a necessary condition for (\ref{eq: delta limit}) to be satisfied  
is $\delta_\e \gtrsim_\e \e^{1-\rho(1+\al)}$. 
We define 
\begin{equation}\label{def: delta eps}
\delta_\e := \e^{1-\rho(1+\al)}|\ln(\e)|^4.
\end{equation}
For the right-hand side to tend to $0$ is equivalent to 
\begin{equation}\label{eq: rho upper bound}
\rho < \frac{1}{\al+1}. 
\end{equation}
In particular for all $\al\in (0,2)$ 
\begin{equation}\label{eq: al times rho}
\al \rho < \frac{\al}{1+\al} < \frac{2}{3}<1.
\end{equation}
Since $\xi^\e$ has exponential moments we can compensate it 
\begin{align*}
\ti \xi^\e_t:= \xi^\e_t - t \EE[\xi^\e_1].
\end{align*}
It is a direct consequence of Lemma 2.1 in \cite{IP08}, 
which treats the same situation, that for any $c>0$
\begin{equation}\label{eq: small noise finite interval}
\PP(\sup_{t\in [0, r^\e]} |\e \ti \xi^\e|> c) \lqq \exp( - \frac{c}{\e^{1-\rho} r^\e}). 
\end{equation}
A small direct calculation or Lemma 3.1 in \cite{IP08} yields that there is constant $h_1>0$ such that 
\[
|\EE[\e\xi^\e_1] | \lqq h_1 \e^{1-\rho}.
\]
The choice of $r^\e$ in (\ref{def: r eps}) and $\rho$ in (\ref{eq: rho upper bound}) 
we obtain that 
\begin{equation}\label{eq: drift is smaller than delta}
h_1 r^\e \e^{1-\rho} = \e^{1-(\al+1)\rho}|\ln(\e)|^2 \lqq_\e  \e^{1-(\al+1)\rho}|\ln(\e)|^4 = \delta_\e. 
\end{equation}
Hence for any $\e>0$ sufficiently small 
we have $|r^\e \EE[\e \xi^\e_1]|\lqq \delta_\e$ and infer 
\begin{align}
\PP(\sup_{t\in [0, T_1]} |\e \xi^\e_t|> c) 
&= \PP(\sup_{t\in [0, T_1]} |\e \ti \xi^\e_t|> c) \nonumber\\
&\lqq \PP(\sup_{t\in [0, r^\e]} |\e \ti \xi^\e_t|> c) + \PP(T_1 > r^\e)\nonumber\\
&\lqq  \exp( - \frac{c}{\e^{1-\rho} r^\e}) + \exp(-\e^{\alpha\rho} r^\e). \label{eq: entire small noise}
\end{align}
Denote $V^{x, \e}_t = Y^{x, \e}_t-c- \e \xi^\e_t$. 
The monotonicity of $b$ on $(0, \infty)$ yields on the events $\{t\in [0, T_1]\}$
and $\{\sup_{t\in [0, T_1]} |\e \xi^\e_s|\lqq 2c\}$ that 
\begin{align*}
V^{x,\e}_t &= x-c + \int_0^t b(V^{x, \e}_s +c +  \e \xi^\e_s) ds \\
&\gqq x-c+ \int_0^t b(V^{x,\e}_s) ds.
\end{align*}
By (\ref{eq: drift is smaller than delta}) we may set $c=\delta_\e$ we obtain
\begin{align*}
V^{x,\e}_t &\gqq x-\delta_\e + \int_0^t b(V^{x,\e}_s) ds \qquad t\in [0, T_1].
\end{align*}
Hence an elementary comparison argument implies under these assumptions 
\[
V^{x,\e}_t \gqq u(t; x-\delta_\e), \qquad \mbox{ for all } \quad t\in [0,T_1], \quad x\gqq \delta_\e.
\]
In particular in the preceding setting we take the supremum over all $x\gqq 4\delta_\e$ and obtain 
\begin{align}
&\sup_{x\in D_{3\delta_\e}^+} \PP(\sup_{t\in [0, T_1]} (Y^{x,\e}_{t} - (u(t; x-\delta_\e)-\delta_\e))<0)\nonumber \\
&\lqq \PP(\sup_{t\in [0, T_1]} |\e \xi^\e|> \delta_\e) 
\lqq  \exp( - \frac{\delta_\e}{\e^{1-\rho} r^\e}) + \exp(-\e^{\alpha\rho} r^\e) = 2\e^2. \label{ineq: y-u<0a}
\end{align}
With the identical reasoning we obtain
\begin{align}
\sup_{x\gqq (i-1) \gamma_\e} \PP(\sup_{t\in [0, T_1]} (Y^{x,\e}_{t} - (u(t; x-\delta_\e)-\delta_\e))<0) 
&\lqq \PP(\sup_{t\in [0, T_1]} |\e \xi^\e|> i \gamma_\e) \nonumber\\
&\lqq  \exp( - \frac{i \gamma_\e}{2\e^{1-\rho} r^\e}) + \exp(-\frac{i \e^{\alpha\rho} r^\e}{2}). \label{ineq: y-u<0b}\\\nonumber
\end{align}
\begin{rems}
In the light of the observations 2.1) and 2.2) it is clear that 
the exit behavior is mainly determined by the behavior of the large jumps $\e W_i$. 
 \end{rems}

\paragraph{3) Estimate of the Laplace transform of the exit time: } 
We estimate the Laplace transform of the first exit time. Let $\theta>0$. 
Then 
\begin{align*}
\sup_{x\in D_{3\delta_\e}^+}\mathbb{E}\Big[e^{-\theta\varepsilon^\alpha \tau^{x,\varepsilon, -}} \Big] 
&= \sum_{k=1}^\infty \sup_{x\in D_{3\delta_\e}^+}\mathbb{E}\Big[e^{-\theta\varepsilon^\alpha \tau^{x, \varepsilon, -}} 
\mathbf{1}\{\tau^{x,\varepsilon,-} \in (T_{k-1}, T_k]\} \Big] \\
&\lqq \sum_{k=1}^\infty \sup_{x\in D_{3\delta_\e}^+}\mathbb{E}\Big[e^{-\theta \varepsilon^\alpha T_{k-1}} 
\mathbf{1}\{\tau^{x,\varepsilon,-} \in (T_{k-1}, T_k]\} \Big] \\
&\lqq \sum_{k=1}^{n_\varepsilon} \sup_{x\in D_{3\delta_\e}^+}\mathbb{E}\Big[e^{-\theta \varepsilon^\alpha T_{k-1}} 
\mathbf{1}\{\tau^{x,\varepsilon,-} \in (T_{k-1}, T_k]\} \Big] 
+ \sum_{k=n_\varepsilon}^{\infty} \mathbb{E}\Big[e^{-\theta \varepsilon^\alpha T_{1}} \Big]^k\\
& =: \sum_{k=1}^{n_\e} \iI_1(k) + \iI_2 =: \iI_1 + \iI_2.
\end{align*}
\textbf{3.1) The infinite remainder: } For the second sum we obtain  
\begin{align}
\iI_2 &= \sum_{k=n_\varepsilon}^{\infty} \Big(\frac{1}{1+ \frac{\theta \varepsilon^\alpha}{\lambda_\varepsilon}}\Big)^k
= \sum_{k=n_\varepsilon}^{\infty} e^{k \ln\Big(1-\frac{\theta \varepsilon^\alpha}{\lambda_\varepsilon}\Big)}
\lesssim_\varepsilon \sum_{k=n_\varepsilon}^{\infty} e^{-k \frac{2\theta \varepsilon^\alpha}{\lambda_\varepsilon}}
= \frac{e^{-n_\varepsilon \frac{2\theta \varepsilon^\alpha}{\lambda_\varepsilon}}}{1-e^{- \frac{2\theta \varepsilon^\alpha}{\lambda_\varepsilon}}}\nonumber\\
&\lesssim_\varepsilon \frac{e^{-n_\varepsilon \frac{2\theta \varepsilon^\alpha}{\lambda_\varepsilon}}}{\frac{2\theta \varepsilon^\alpha}{\lambda_\varepsilon}}
= e^{-n_\varepsilon \frac{2\theta \varepsilon^\alpha}{\lambda_\varepsilon} - \ln(\frac{2\theta \varepsilon^\alpha}{\lambda_\varepsilon})} 
=: S_1(\varepsilon). \label{eq: minimal growth n}
\end{align}
In order to get $S_1(\e)\ra 0$ as $\e\ra 0$, 
we need the asymptotics  
\begin{equation}\label{eq: asymptotics for n}
n_\e \e^{\alpha(1-\rho)}+\ln(\e) \ra \infty,
\end{equation}
or for simplicity 
\[n_\e \gtrsim_\e \frac{1}{\e^{\al(1-\rho)}}+|\ln(\e)|.\]
If we define 
\begin{equation} \label{def: n eps}
n_\e := \frac{|\ln(\e)|^2}{\e^{\al(1-\rho)}}, 
\end{equation}
we obtain 
\begin{align*}
S_1(\e) \approx_\e \e^{2+\al(1-\rho)}\ra 0\qquad \mbox{ as }\e \ra 0.
\end{align*}

\paragraph{3.2) Estimate of the main sum: } 
The rest of the proof is devoted to estimate $\sum_{k=0}^{n_\e} \iI_1(k)$. 
We define the following events for $y\in D_{3\delta_\e}^+$ and $s, t\gqq 0$ by 
\begin{align*}
A_{t,s,y}^- &:= \{X_{r}^{\cdot, \varepsilon} \circ \theta_{s}(y)\in D_{\delta_\varepsilon}^+ \mbox{ for all } r\in [0, t]\},\\
B_{t,s,y}^- &:= \{X_{r}^{\cdot, \varepsilon} \circ \theta_{s}(y)\in D_{\delta_\varepsilon}^+ \mbox{ for all } r\in [0, t) \mbox{ and }
X_{t}^{\cdot, \varepsilon} \circ \theta_{s}(y) \notin D^+_{\delta_\varepsilon}\}.
\end{align*}
Recall the waiting times $t_k := T_k - T_{k-1}$ and exploit the decomposition
\begin{align*}
\{\tau^{x, \varepsilon, -}  \in (T_{k-1}, T_k] \} 
&= \bigcap_{i=1}^{k-1} A^-_{t_i, T_{i-1}, X_{T_{i-1}, x}} \cap \Big(\bigcup_{t\in (0, t_k]} B_{t, T_{k-1}, x}^-\Big).
\end{align*}

\paragraph{3.2.1) Derivation of the recursion for the idealized exit from an unstable point $0$: } 
We estimate $\iI_1(k)$ with the help of the strong Markov property 
\begin{align*}
&\iI_1(k) \\
&= \sup_{x\in D_{3\delta_\e}^+}\mathbb{E}\Big[\mathbb{E}\Big[\prod_{i=1}^{k-1} e^{-\theta\lambda_\varepsilon t_i} 
\mathbf{1}\Big(A^-_{t_i, T_{i-1}, X_{T_{i-1}, x}}\Big) \\
&\qquad \Big(
 \mathbf{1}\big\{u(T_1;x-\delta_\e)-\delta_\e+ \varepsilon W_1 > \la_\e^{-\frac{1}{\Gamma(1-\beta)}}\big\}+ 
(\mathbf{1}\big\{u(T_1;x-\delta_\e)-\delta_\e+ \varepsilon W_1 \lqq \la_\e^{-\frac{1}{\Gamma(1-\beta)}}\big\}\Big) \\
&\qquad\Big(\mathbf{1}\big\{\sup_{t\in [0, T_1]} (Y^{x, \varepsilon, 1}_t - (u(t;x-\delta_\e)-\delta_\e)) \gqq 0 \big\} +
 \mathbf{1}\big\{\sup_{t\in [0, T_1]} (Y^{x, \varepsilon, 1}_t - (u(t;x-\delta_\e)-\delta_\e))< 0 \big\}\Big)\\
&\qquad \qquad \qquad \mathbf{1}\big(\bigcup_{t\in (0, T_{k}- T_{k-1}]} B_{t, T_{k-1}, x}^-\big)~|~ \mathcal{F}_{T_1}\Big]\Big]\\
&\lqq \sup_{y\in D_{3\delta_\e}^+} \mathbb{E}\Big[e^{-\theta \lambda_\varepsilon T_1} \mathbf{1}\Big(A^-_{T_1 , 0, y}\Big)
\mathbf{1}\big\{\sup_{t\in [0, T_1]} (Y^{y, \varepsilon, 1}_t - (u(t;y-\delta_\e)-\delta_\e))\gqq 0\big\}
\Big]\\
&\qquad \sup_{y\gqq \gamma_\varepsilon} \mathbb{E}\Big[\prod_{i=1}^{k-1} e^{-\theta\lambda_\varepsilon t_i} 
\mathbf{1}\Big(A^-_{t_i, T_{i-1}, X_{T_{i-1}, y}}\Big)
\mathbf{1}\Big(\bigcup_{t\in (0, T_{k}- T_{k-1}]} B_{t, T_{k-1}, y}^-\Big)\Big]\\
&\qquad + \sup_{y\in D_{3\delta_\e}^+} \mathbb{E}\Big[e^{-\theta \lambda_\varepsilon T_1} 
\mathbf{1}\big\{\sup_{t\in [0, T_1]} (Y^{y, \varepsilon, 1}_t - (u(t;y-\delta_\e)-\delta_\e))<0 \big\}
\Big] \\
&\qquad + \sup_{y\in D_{3\delta_\e}^+} \mathbb{E}\Big[e^{-\theta \lambda_\varepsilon T_1} 
\mathbf{1}\Big\{u(T_1;y-\delta_\e)-\delta_\e+ \varepsilon W_1 \lqq \la_\e^{-\frac{1}{\Gamma(1-\beta)}}\Big\}\Big]
\end{align*}
where we recall that 
$\gamma_\varepsilon = (\la_\e^{-\frac{1}{\Gamma}} - (3\delta_\e)^{1-\beta})^{\frac{1}{1-\beta}}$. 
Taking a closer look we may identify the preceding inequality as the recursive estimate 
\begin{align}
&\sup_{x\in D_{3\delta_\e}^+} \mathbb{E}\Big[e^{-\theta\lambda_\varepsilon T_{k-1}} 
\mathbf{1}\{\tau^{x, \varepsilon, - } \in (T_{k-1}, T_k]\}\Big] \nonumber\\
&\lqq \sup_{y\gqq \gamma_\varepsilon} 
\mathbb{E}\Big[e^{-\theta\lambda_\varepsilon T_{k-2}} \mathbf{1}\{\tau^{y, \varepsilon, - } \in (T_{k-2}, T_{k-1}]\}\Big] \nonumber\\
&\qquad\cdot \sup_{y\in D_{3\delta_\e}^+} \mathbb{E}\Big[e^{-\theta \lambda_\varepsilon T_1} \mathbf{1}\Big(A^-_{T_1 , 0, y}\Big)
\mathbf{1}\big\{\sup_{t\in [0, T_1]} (Y^{y, \varepsilon, 1}_t - (u(t;y-\delta_\e)-\delta_\e))\gqq 0 \big\}
\Big] \nonumber\\
&\qquad + \sup_{y\in D_{3\delta_\e}^+} \mathbb{E}\Big[e^{-\theta \lambda_\varepsilon T_1} 
\mathbf{1}\big\{\sup_{t\in [0, T_1]} (Y^{y, \varepsilon, 1}_t - (u(t;y-\delta_\e)-\delta_\e))<0 \big\}
\Big]  \nonumber\\
&\qquad + \sup_{y\in D_{3\delta_\e}^+} \mathbb{E}\Big[e^{-\theta \lambda_\varepsilon T_1} 
\mathbf{1}\Big\{u(T_1;y-\delta_\e)-\delta_\e+ \varepsilon W_1 \lqq \la_\e^{-\frac{1}{\Gamma(1-\beta)}}\Big\}\Big].\label{eq: recursion i=1}
\end{align}
The same reasoning yields for all $2 \lqq i \lqq k$ the recursive inequality
\begin{align*}
&\sup_{x \gqq (i-1) \gamma_\varepsilon} \mathbb{E}\Big[e^{-\theta\lambda_\varepsilon T_{k-1}} 
\mathbf{1}\{\tau^{x, \varepsilon, - } \in (T_{k-1}, T_k]\}\Big] \\
&\lqq \sup_{y\gqq i \gamma_\varepsilon} 
\mathbb{E}\Big[e^{-\theta\lambda_\varepsilon T_{k-2}} \mathbf{1}\{\tau^{y, \varepsilon, - } \in (T_{k-2}, T_{k-1}]\}\Big] \\
&\qquad \cdot \sup_{y\gqq (i-1) \gamma_\varepsilon} \mathbb{E}\Big[e^{-\theta \lambda_\varepsilon T_1} 
\mathbf{1}\Big(A^-_{T_1 , 0, y}\Big)
\mathbf{1}\big\{\sup_{t\in [0, T_1]} (Y^{y, \varepsilon, 1}_t - (u(t;y-\delta_\e)-\delta_\e))\gqq 0\big\}
\Big]\\
&\qquad + \sup_{y\gqq (i-1) \gamma_\varepsilon} \mathbb{E}\Big[e^{-\theta \lambda_\varepsilon T_1} 
\mathbf{1}\big\{\sup_{t\in [0, T_1]} (Y^{y, \varepsilon, 1}_t - (u(t;y-\delta_\e)-\delta_\e))<0\big\}\Big]\\
&\qquad + \sup_{y\gqq (i-1) \gamma_\varepsilon} \mathbb{E}\Big[e^{-\theta \lambda_\varepsilon T_1} 
\mathbf{1}\Big\{u(T_1;y-\delta_\e)-\delta_\e + \varepsilon W_1 \lqq \la_\e^{-\frac{1}{\Gamma(1-\beta)}}\Big\}\Big].
\end{align*}
Hence solving the recursion we obtain 
\begin{align}
\iI_1(k) &\lqq\prod_{j=1}^{k-1} \sup_{y\gqq (j-1)\gamma_\varepsilon\vee 3\delta_\e} 
\mathbb{E}\Big[e^{-\theta \lambda_\varepsilon T_1} 
\mathbf{1}\Big(A^-_{T_1 , 0, y}\Big)
\mathbf{1}\big\{\sup_{t\in [0, T_1]} (Y^{y, \varepsilon, 1}_t - (u(t;y-\delta_\e)-\delta_\e))\gqq 0\big\}
\Big]\nonumber\\
&\qquad \cdot  \sup_{y \gqq (k-1) \gamma_\varepsilon\vee 3\delta_\e} \mathbb{P}\Big(\tau^{y, \varepsilon, - } \in (0, T_1]\Big)\nonumber\\
&\qquad + \sum_{i=1}^{k-2} \sup_{y\in (i-1) \gamma_\e \vee 3\delta_\e} \mathbb{E}\Big[e^{-\theta \lambda_\varepsilon T_1} 
\mathbf{1}\big\{\sup_{t\in [0, T_1]} (Y^{y, \varepsilon, 1}_t - (u(t;y-\delta_\e)-\delta_\e))<0\big\}
\Big] \nonumber\\
&\qquad + \sum_{i=1}^{k-2} \sup_{y\gqq (i-1) \gamma_\varepsilon\vee 3\delta_\e} \mathbb{E}\Big[e^{-\theta \lambda_\varepsilon T_1} 
\mathbf{1}\Big\{u(T_1;y-\delta_\e)-\delta_\e + \varepsilon W_1 \lqq \la_\e^{-\frac{1}{\Gamma(1-\beta)}}\Big\}\Big].\label{ineq: recursion}
\end{align}
\paragraph{3.2.2) Estimate of the second sum of the recursion (\ref{ineq: recursion}): } 
By (\ref{ineq: y-u<0a}) and (\ref{ineq: y-u<0b}) there exists $\e_0\in (0,1)$ such that for $\e \in (0,\e_0]$ 
\begin{align*}
& \sum_{k=1}^{n_\varepsilon} 
\sum_{i=1}^{k-2} \sup_{y\in (i-1) \gamma_\e \vee 3\delta_\e} \mathbb{E}\Big[e^{-\theta \lambda_\varepsilon T_1} 
\mathbf{1}\big\{\sup_{t\in [0, T_1]} (Y^{y, \varepsilon, 1}_t - (u(t;y-\delta_\e)-\delta_\e))<0\big\}
\Big]\\
&\lqq n_\e \sum_{i=1}^{\infty} \sup_{y\gqq (i-1) \gamma_\e \vee 3\delta_\e} \mathbb{E}\Big[e^{-\theta \lambda_\varepsilon T_1} 
\mathbf{1}\big\{\sup_{t\in [0, T_1]} (Y^{y, \varepsilon, 1}_t - (u(t;y-\delta_\e)-\delta_\e))<0\big\}
\Big]\\
&\lqq 2 n_\e (\exp( - \frac{\delta_\e}{\e^{1-\rho} r^\e}) + \exp(-\e^{\alpha\rho} r^\e)) 
=: S_2(\varepsilon) \searrow 0,
\end{align*}
with the convention $\sum^{-1} = 0$. We determine the order of $S_2$ 
\begin{align*}
&n_\e (\exp( - \frac{\delta_\e}{\e^{1-\rho} r^\e}) + \exp(-\e^{\alpha\rho} r^\e)) \\
&= |\ln(\e)|^2 \e^{-\al(1-\rho)} \exp(-\frac{\e^{1-\rho(1+\al)} |\ln(\e)|^4}{\e^{1-\rho} \e^{-\al\rho} |\ln(\e)|^2}) 
+ |\ln(\e)|^2 \e^{-\al(1-\rho)} \exp(-\e^{-\al\rho} |\ln(\e)|^2 \e^{\al\rho}) \\
&= 2 |\ln(\e)|^2 \e^{2 -\al + \al\rho}.
\end{align*}

\paragraph{3.2.3) Estimate of the third sum in the recursion (\ref{ineq: recursion}): } 
For $i=0$ and $0 <\e\lqq \e_0$ 
we perform the core calculation of the article. 
The idea is the following: $X^{x, \e}_t \gtrsim_\e u(t;x-\delta) -2 \delta_\e + \e W_1 \ind\{t= T_1\}$ for all $t\in [0,T_1]$. 
For small $\e$ and $3\delta_\e< x \lqq \gamma_\e$ 
the solution $u(T_1,x-\delta_\e)-\delta_\e$ escapes sufficiently far away from $x$, that is $u(T_1,x-\delta_\e)-\delta_\e \gqq 2\gamma_\e$, 
such that the probability that $u(T_1, x-\delta_\e)-\delta_\e + \e W_1 < \gamma_\e$ decays sufficiently fast. 

\paragraph{3.2.3.1) Estimate of the backbone decomposition of the first exit event: }
Due to the independence of $T_1$ and $W_1$ we may calculate 
for $\gamma^*_\e(x) = \frac{(2\gamma_\e+\delta_\e)^{1-\beta}-(x-\delta_\e)^{1-\beta}}{B (1-\beta)}$
\begin{align}
& \sup_{3\delta_\e < x\lqq \gamma_\e} \mathbb{E}\Big[e^{-\theta \lambda_\varepsilon T_1} 
\mathbf{1}\Big(u(T_1;x-\delta_\e)-\delta_\e+ \varepsilon W_1 \lqq \gamma_\e \Big)\Big]\nonumber\\
& \lqq \sup_{3\delta_\e < x\lqq \gamma_\e} \mathbb{E}\Big[e^{-\theta \lambda_\varepsilon T_1} 
\mathbf{1}\big(u(T_1;x-\delta_\e)-\delta_\e+ \varepsilon W_1 \lqq \gamma_\e \big)
 \ind\big(u(T_1;x-\delta_\e) > 2\gamma_\e+\delta_\e\big)  \Big]\nonumber \\
&\qquad + \sup_{3\delta_\e < x \lqq \gamma_\e}\PP(u(T_1; x-\delta_\e) \lqq 2\gamma_\e+\delta_\e)\nonumber\\
& = \sup_{3\delta_\e < x\lqq \gamma_\e} \int_{\gamma_\e^*(x)}^\infty \PP(u(t;x-\delta_\e)-\delta_\e+ \e W_1 \lqq \gamma_\e)
 \la_\e e^{-\la_\e t} dt  
+ \sup_{3\delta_\e < x\lqq \gamma_\e}\PP(u(T_1; x-\delta_\e) \lqq 2\gamma_\e+\delta_\e)\label{ineq: second sum main term}
\end{align}
The second term is known from (\ref{eq: det escape from 0}) and tends to $0$. 
It remains to calculate the first one. 
\begin{align}
& \sup_{3\delta_\e < x\lqq \gamma_\e}\int_{\gamma_\e^*(x)}^\infty 
\PP(u(t;x-\delta_\e)-\delta_\e+ \e W_1 \lqq \gamma_\e) \la_\e e^{-\la_\e t} dt \nonumber\\
& = \sup_{3\delta_\e < x\lqq \gamma_\e}
\int_{\gamma_\e^*(x)}^\infty \nu((-\infty, \frac{1}{\e}(\gamma_\e - (u(t;x-\delta_\e)-\delta_\e)]) e^{-\la_\e t} dt\nonumber\\
& =\sup_{3\delta_\e < x\lqq \gamma_\e} \int_{\gamma_\e^*(x)}^\infty 
\nu((-\infty, \frac{1}{\e}(\gamma_\e +\delta_\e - (B(1-\beta) t + (x-\delta_\e)^{1-\beta})^{\frac{1}{1-\beta}})]) e^{-\la_\e t} dt\nonumber\\
& = \sup_{3\delta_\e < x\lqq \gamma_\e}\frac{\alpha}{4} \frac{\e^\alpha}{\la_\e} 
\int_{\gamma_\e^*(x)}^\infty \frac{1}{((B(1-\beta) t + (x-\delta_\e)^{1-\beta})^{\frac{1}{1-\beta}}-(\gamma_\e+\delta_\e))^\alpha }  
\la_\e e^{-\la_\e t} dt\nonumber\\
& \lqq \frac{\alpha c^-}{2} \frac{\e^\alpha}{\la_\e} \frac{1}{\gamma_\e^\al}.\label{ineq: second sum i=0}
\end{align}
In the last step we have used the fact that the integrand is monotonically decreasing in the variable $t$ and weight $c^-$ of 
of the negative branch of the L\'evy measure. The term 
\begin{align*}
\frac{\e^\alpha}{\la_\e} \frac{1}{\gamma_\e^\alpha} 
\approx_\e \e^{\alpha(1- \rho) + \frac{\alpha^2\rho}{\Gamma(1-\beta)}}, 
\end{align*}
converges to $0$ as $\e \ra 0$. This gives an estimate for the last term in (\ref{eq: recursion i=1}). 
The last term in (\ref{ineq: recursion}) deals with initial values $(i-1) \gamma_\e < x\lqq i \gamma_\e$. 
We obtain for 
\[\gamma_\e^*(i,x) := \frac{((i+1)\gamma_\e + \delta_\e)^{1-\beta}-x^{1-\beta}}{B (1-\beta)}\]
with the analogous calculations the following estimate
\begin{align}
& \sup_{(i-1) \gamma_\e < x\lqq i \gamma_\e}\mathbb{E}\Big[e^{-\theta \lambda_\varepsilon T_1} 
\mathbf{1}\Big(u(T_1;x-\delta_\e)-\delta_\e+ \varepsilon W_1 \lqq \gamma_\e \Big)\Big] \nonumber\\
& = \sup_{(i-1) \gamma_\e < x\lqq i \gamma_\e}
\frac{\alpha c^-}{2} \frac{\e^\alpha}{\la_\e} 
\int_{\gamma_\e^*(i,x)}^\infty \frac{1}{((B(1-\beta) t + (x-\delta_\e)^{1-\beta})^{\frac{1}{1-\beta}}-(\gamma_\e-\delta_\e))^\alpha }
 \la_\e e^{-\la_\e t} dt\nonumber\\
&\qquad + \sup_{(i-1) \gamma_\e < x\lqq i \gamma_\e}\PP(u(T_1; x-\delta_\e) \lqq (i+1)\gamma_\e+\delta_\e)\nonumber\\
& \lqq \frac{\alpha c^-}{2} \frac{\e^\alpha}{\la_\e} \frac{1}{\gamma_\e^\al i^\al}
+ \sup_{(i-1) \gamma_\e < x\lqq i \gamma_\e}\PP(u(T_1; x-\delta_\e) \lqq (i+1)\gamma_\e+\delta_\e).\label{ineq: second sum i>0}
\end{align}
Combining the estimates (\ref{ineq: second sum main term}), (\ref{ineq: second sum i=0}) and (\ref{ineq: second sum i>0}) we obtain for any 
$C>1$ 
\begin{align}
&\sum_{i=1}^{k-2} \sup_{y\gqq (i-1) \gamma_\varepsilon\vee 3\delta_\e} \mathbb{E}\Big[e^{-\theta \lambda_\varepsilon T_1} 
\mathbf{1}\Big\{u(T_1;y-\delta_\e)-\delta_\e+ \varepsilon W_1 \lqq \la_\e^{-\frac{1}{2(1-\beta)}}\Big\}\Big]\nonumber\\
&= \sum_{i=1}^{k-2} \sup_{j\gqq i} \sup_{(j-1) \gamma_\varepsilon\vee 3\delta_\e< y \lqq j \gamma_\e } 
\mathbb{E}\Big[e^{-\theta \lambda_\varepsilon T_1} 
\mathbf{1}\Big\{u(T_1;y-\delta_\e)-\delta_\e+ \varepsilon W_1 \lqq \la_\e^{-\frac{1}{2(1-\beta)}}\Big\}\Big]\nonumber\\
&\lesssim_\e \sum_{i=1}^{k-2} \sup_{j\gqq i} \Big(\frac{\alpha c^-}{2} \frac{\e^\alpha}{\la_\e} \frac{1}{\gamma_\e^\al j^\al} 
+ \sup_{(j-1) \gamma_\e < x\lqq j \gamma_\e}\PP(u(T_1; x-\delta_\e) \lqq (j+1)\gamma_\e+\delta_\e)\Big)\nonumber\\
&\lesssim_\e \sum_{i=1}^{k-2} \Big(\frac{\alpha c^-}{2} \frac{\e^\alpha}{\la_\e} \frac{1}{\gamma_\e^\al i^\al} 
+ C(1- \exp(-[(i+1)^{1-\beta} - i^{1-\beta}] \frac{\gamma_\e^{1-\beta} \la_\e}{B(1-\beta)})\Big)\nonumber\\
&\lqq \sum_{i=1}^{k-2} \Big(\frac{\alpha c^-}{2} \frac{\e^\alpha}{\la_\e} \frac{1}{\gamma_\e^\al i^\al} 
+ C[(i+1)^{1-\beta} - i^{1-\beta}] \frac{\gamma_\e^{1-\beta} \la_\e}{B(1-\beta)}\Big)\nonumber\\
&\lqq \sum_{i=1}^{k-2} \Big(\frac{\alpha c^-}{2} \frac{\e^\alpha}{\la_\e} \frac{1}{\gamma_\e^\al i^\al} 
+ \frac{C}{B} \frac{\gamma_\e^{1-\beta} \la_\e}{i^\beta}\Big)\nonumber\\
&=  \frac{\alpha c^-}{2} \frac{\e^\alpha}{\la_\e} \frac{1}{\gamma_\e^\al} \sum_{i=1}^{k-2} 
\frac{1}{i^\al} 
+ \frac{C}{B} \gamma_\e^{1-\beta} \la_\e \sum_{i=1}^{k-2} \frac{1}{i^\beta}\nonumber \\
&\lqq C \frac{\alpha c^-}{2} \frac{\e^\alpha}{\la_\e} \frac{1}{\gamma_\e^\al} k^{1-\al} + \frac{C}{B} \gamma_\e^{1-\beta} \la_\e k^{1-\beta}.
\label{eq: backbone}
\end{align}
Hence we may sum up
\begin{align}
& \sup_{3\delta_\e < x\lqq \gamma_\e} \mathbb{E}\Big[e^{-\theta \lambda_\varepsilon T_1} 
\mathbf{1}\Big(u(T_1;x-\delta_\e)-\delta_\e+ \varepsilon W_1 \lqq \gamma_\e \Big)\Big]\nonumber\\
&+ \sum_{k=2}^{n_\varepsilon} \sum_{i=1}^{k-2} \sup_{y\gqq (i-1) \gamma_\varepsilon\vee 3\delta_\e} 
\mathbb{E}\Big[e^{-\theta \lambda_\varepsilon T_1} 
\mathbf{1}\Big\{u(T_1;y-\delta_\e)-\delta_\e+ \varepsilon W_1 \lqq \la_\e^{-\frac{1}{\Gamma(1-\beta)}}\Big\}\Big]\nonumber\\
&\lesssim_\e \frac{\e^\alpha}{\la_\e} \frac{1}{\gamma_\e^\alpha} + \la_\e^{1-\frac{1}{\Gamma}}
+ \frac{C c^- \alpha}{2} \frac{\e^\alpha}{\la_\e} \frac{1}{\gamma_\e^\al} (n_\e)^{2-\al} 
+ \frac{C}{B(1-\beta)} \gamma_\e^{1-\beta} \la_\e (n_{\e})^{2-\beta} =: S_3(\e)\label{eq: S3}
\end{align}

\paragraph{3.2.3.2) Conditions on parameters in order to establish the convergence $S_3(\e) \ra 0$: } 
\begin{itemize}
 \item We check the order of the second to last expression on the right-hand side 
\begin{align*}
\e^{\alpha(1- \rho(1 - \frac{\alpha}{\Gamma(1-\beta)}))} n_\e^{2-\alpha} 
&\approx_\e \e^{\alpha(1- \rho(1 - \frac{\alpha}{\Gamma(1-\beta)}))-\al (2-\alpha) (1-\rho)} |\ln(\e)|^{2(2-\al)} \\
&= \e^{\alpha[(1- \rho(1 - \frac{\alpha}{\Gamma(1-\beta)}))- (2-\alpha) (1-\rho)]} |\ln(\e)|^{2(2-\al)}.
\end{align*}
The essential sign of the exponent hence is given as the sign of 
\begin{equation}\label{eq: sign}
(1- \rho) + \frac{\rho\alpha}{\Gamma(1-\beta)}- (2-\alpha) (1-\rho) = (\al-1)(1-\rho) + \frac{\rho\alpha}{\Gamma(1-\beta)}. 
\end{equation}
\begin{itemize}
 \item For $1\lqq \al<2$ the sign is positive, since all terms are nonnegative and the last term is positive. 
 \item For $0< \al < 1$ we calculate that the positivity of (\ref{eq: sign}) 
\begin{align*}
0&< -(1-\al)(1-\rho) + \frac{\rho\alpha}{2(1-\beta)} = -(1-\al) + \rho [\frac{\alpha}{2(1-\beta)}+(1-\al)]\\
\end{align*}
is equivalent to 
\begin{align*}
\rho_0(\al, \beta) := \frac{\Gamma(1-\al)(1-\beta)}{\Gamma(1-\al)(1-\beta) + \al} < \rho
\end{align*}
where the right-hand side is strictly less than $1$. 
Hence in this case the sign is positive if we choose $\rho_0 < \rho <1$. 
\end{itemize}
 \item For the second expression on the right-hand side we obtain 
\begin{align*}
\gamma_\e^{1-\beta} \la_\e (n_{\e})^{2-\beta} 
&\approx_\e \e^{-(1-\beta) \frac{\al\rho}{\Gamma(1-\beta)}} \e^{\al\rho} \e^{-\al(1-\rho)(2-\beta)} |\ln(\e)|^{2-\beta}
= \e^{\al\rho(1-\frac{1}{\Gamma}) - \al(1-\rho)(1-\beta)}|\ln(\e)|^{2-\beta}.
\end{align*}
The positivity of the exponent depends on the sign of 
\begin{align*}
0< (1-\frac{1}{\Gamma})\rho - (1-\rho)(1-\beta) = \rho ((1-\frac{1}{\Gamma}) + (1-\beta)) - (1-\beta),
\end{align*}
which is equivalent to 
\begin{align*}
\rho > \frac{(1-\frac{1}{\Gamma})(1-\beta)}{(1-\frac{1}{\Gamma})(1-\beta) + 1}=: \rho_1(\beta).
\end{align*}
Since $\rho_1(\beta)<1$ for all $\rho_1 < \rho <1$ the second exponent is also positive. 
\end{itemize}

\paragraph{3.2.3.3) Verify the compatibility of the choice of convergent parameters: } 

We check that the parameters $\beta$ and $\alpha$ are compatible with 
$\rho < \frac{1}{1+\al}$ in (\ref{eq: rho upper bound}), which ensures 
that $\delta_\e\ra 0$, as $\e \ra 0$.  
The first convergence in (\ref{eq: S3}) yields 
\begin{align*}
\rho_0 &= \frac{\Gamma(1-\al)(1-\beta)}{\Gamma(1-\al)(1-\beta) + \al} < \frac{1}{1+\al}
\qquad \vzv \quad \frac{\Gamma-1-\Gamma\beta}{\Gamma(1-\beta)}<  \al, 
\end{align*}
where the left hand side is negative, since $\Gamma < \frac{1}{1-\beta}$. 
Hence it does not impose any additional restriction on $\al$. 
The second condition yields 
\begin{align*}
&\rho_1 = \frac{(1-\frac{1}{\Gamma})(1-\beta)}{(1-\frac{1}{\Gamma})(1-\beta) + 1} < \frac{1}{1+\al}
\quad \vzv \quad \frac{1}{(1-\frac{1}{\Gamma})(1-\beta)} > \al, 
\end{align*}
In order to get rid of any restriction on $\al$ we calculate 
\begin{align*}
2 \lqq \frac{1}{(1-\frac{1}{\Gamma})(1-\beta)} \quad \vzv \quad 
\Gamma \lqq \frac{2(1-\beta)}{2(1-\beta) - 1}. 
\end{align*}
We can always choose 
\begin{align}
\Gamma &:= \frac{1}{2} \Big(1+ \frac{1}{2}\Big(\frac{1}{1-\beta}+ \frac{2(1-\beta)}{2(1-\beta) - 1}\Big)\Big)\label{eq: choice of Gamma}\\
\rho &:= \frac{1}{2}\Big(\rho_1(\beta) + \frac{1}{1+\al}\Big),\label{eq: choice of rho}
\end{align}
satisfying all conditions required before. 

\paragraph{3.2.4) Estimate of the first sum of the recursion (\ref{ineq: recursion}): } 
It remains to estimate the expression 
\begin{align*}
&\sum_{k=1}^{n_\e} \prod_{j=1}^{k-1} \sup_{y\gqq (j-1)\gamma_\varepsilon\vee 3\delta_\e} 
\mathbb{E}\Big[e^{-\theta \lambda_\varepsilon T_1} 
\mathbf{1}\Big(A^-_{T_1 , 0, y}\Big)
\mathbf{1}\big\{\sup_{t\in [0, T_1]} (Y^{y, \varepsilon, 1}_t - (u(t;y-\delta_\e)-\delta_\e))\gqq 0\big\}
\Big]\\
&\qquad \cdot  \sup_{y \gqq (k-1) \gamma_\varepsilon\vee 3\delta_\e} \mathbb{P}\Big(\tau^{y, \varepsilon, - } \in (0, T_1]\Big).\nonumber
\end{align*}
\paragraph{3.2.4.1) We estimate the factors one by one: }
For $j\gqq 2$ 
\begin{align}
& \sup_{y\gqq (j-1)\gamma_\varepsilon\vee 3\delta_\e} \mathbb{E}\Big[\mathbf{1}\Big(A^-_{T_1 , 0, y}\Big) 
\mathbf{1}\big\{\inf_{t\in [0, T_1]} (Y^{y, \varepsilon, 1}_t - (u(t;y-\delta_\e)-\delta_\e))\gqq 0 \big\}
\Big]\nonumber\\
&\qquad \lesssim_\varepsilon 1-(1-C) \mathbb{P}(\varepsilon W_1  < -(j-1)\gamma_\varepsilon)\nonumber\\[2mm]
&\qquad  = 1- \frac{(1-C)}{2}\Big(\frac{\varepsilon}{(j-1)\gamma_\varepsilon}\Big)^{\alpha\rho}\label{eq: non-exit event j>1}
\end{align}
and for $j=1$ 
\begin{align}
& \sup_{y\in D_{3\delta_\e}^+} \mathbb{E}\Big[\mathbf{1}\Big(A^-_{T_1 , 0, y}\Big) 
\mathbf{1}\big\{\inf_{t\in [0, T_1]} (Y^{y, \varepsilon, 1}_t - (u(t;y-\delta_\e)-\delta_\e))\gqq 0 \big\}
\Big]\nonumber\\
& \lqq \sup_{y\in D_{3\delta_\e}^+} \mathbb{E}\Big[\mathbf{1}\Big(A^-_{T_1 , 0, y}\Big) 
\mathbf{1}\big\{\inf_{t\in [0, T_1]} (Y^{y, \varepsilon, 1}_t - (u(t;y-\delta_\e)-\delta_\e))\gqq 0 \big\}
\mathbf{1}\{u(T_1,y-\delta_\e) \gqq 2\gamma_\e\}\Big] \nonumber\\
&\qquad + \sup_{y\in D_{3\delta_\e}^+} \PP(u(t;y) \lqq 2\gamma_\e+\delta_\e)\nonumber\\
&\qquad \lesssim_\varepsilon 1-(1-C) \mathbb{P}(\varepsilon W_1  < -\gamma_\varepsilon) + C \la_\e^{1-\frac{1}{\Gamma}}\nonumber\\[2mm]
&\qquad \lesssim_\e 1- \frac{(1-C)}{2}\Big(\frac{\varepsilon}{\gamma_\varepsilon}\Big)^{\alpha\rho} + C\e^{\al\rho(1+\frac{1}{\Gamma})}.
\label{eq: non-exit event j=1}
\end{align}
We estimate for $k\gqq 2$ with the help of (\ref{ineq: y-u<0a})
\begin{align}
& \sup_{y \gqq (k-1) \gamma_\varepsilon} \mathbb{P}\Big(\tau^{x, \varepsilon, - } \in (0, T_1]\Big)\nonumber\\
&\lqq  \mathbb{P}\Big(W_1 < -(k-1) \frac{\gamma_\varepsilon}{\varepsilon}\Big) 
+ \sup_{y \gqq (k-1) \gamma_\varepsilon} \mathbb{P}\Big(\sup_{t\in [0, T_1]} (Y^{y, \varepsilon, 1}_t-(u(t;y)-\delta_\e))> 0\Big)\nonumber\\
& \lqq \frac{1}{2}\Big(\frac{\varepsilon}{\gamma_\varepsilon}\Big)^{\alpha\rho} \frac{1}{(k-1)^{\alpha\rho}}
+ \sup_{y \in D_{3\delta_\e}^+} \mathbb{P}\Big(\sup_{t\in [0, T_1]} (Y^{y, \varepsilon, 1}_t-(u(t;y)-\delta_\e))> 0\Big)\\
& \lqq \frac{1}{2}\Big(\frac{\varepsilon}{\gamma_\varepsilon}\Big)^{\alpha\rho} \frac{1}{(k-1)^{\alpha\rho}}
+ 2\e^2\label{eq: exit event k>1}
\end{align}
whereas for $k=1$
\begin{align}
&\sup_{y \in D_{3\delta_\e}^+} \mathbb{P}\Big(\tau^{x, \varepsilon, - } \in (0, T_1]\Big)\nonumber\\
& \lqq \frac{1}{2}\Big(\frac{\varepsilon}{\gamma_\varepsilon}\Big)^{\alpha\rho} 
+ \sup_{y \in D_{3\delta_\e}^+} \mathbb{P}\Big(\sup_{t\in [0, T_1]} (Y^{y, \varepsilon, 1}_t-(u(t;y)-\delta_\e))> 0\Big) 
+ \sup_{y\in D_{3\delta_\e}^+} \mathbb{P}(u(T_1, y)\lqq \gamma_\varepsilon +\delta_\e)\nonumber\\
& \lqq \frac{1}{2}\Big(\frac{\varepsilon}{\gamma_\varepsilon}\Big)^{\alpha\rho} 
+ \sup_{y \in D_{3\delta_\e}^+} \mathbb{P}\Big(\sup_{t\in [0, T_1]} (Y^{y, \varepsilon, 1}_t-(u(t;y)-\delta_\e))> 0\Big) 
+ \frac{2}{B(1-\beta)} \lambda_\varepsilon^{1-\frac{1}{\Gamma}},\label{eq: exit event k=1}
\end{align}
where the last term is known from (\ref{eq: det escape from 0}). 

\paragraph{3.2.4.2) Estimate of the entire sum: } 
Collecting the previous (\ref{eq: non-exit event j>1}), (\ref{eq: non-exit event j=1}), (\ref{eq: exit event k>1}), (\ref{eq: exit event k=1}) 
and for the small noise estimate (\ref{ineq: y-u<0b}) together with (\ref{eq: backbone}) we continue 
\begin{align*}
&\sum_{k=1}^{n_\e} \prod_{j=1}^{k-1} \sup_{y\gqq (j-1)\gamma_\varepsilon\vee 3\delta_\e} 
\mathbb{E}\Big[e^{-\theta \lambda_\varepsilon T_1} 
\mathbf{1}\Big(A^-_{T_1 , 0, y}\Big)
\mathbf{1}\big\{\sup_{t\in [0, T_1]} (Y^{y, \varepsilon, 1}_t - (u(t;y-\delta_\e)-\delta_\e))\gqq 0\big\}
\Big]\\
&\qquad \cdot  \sup_{y \gqq (k-1) \gamma_\varepsilon\vee 3\delta_\e} \mathbb{P}\Big(\tau^{y, \varepsilon, - } \in (0, T_1]\Big)\nonumber\\
&\lqq \frac{1}{2}\Big(\frac{\varepsilon}{\gamma_\varepsilon}\Big)^{\alpha\rho} 
+ \sup_{y \in D_{3\delta_\e}^+} \mathbb{P}\Big(\sup_{t\in [0, T_1]} (Y^{y, \varepsilon, 1}_t-(u(t;y)-\delta_\e))> 0\Big) 
+ \frac{2}{B(1-\beta)} \lambda_\varepsilon^{1-\frac{1}{\Gamma}}\\
&\qquad + \frac{1}{2}\Big(\frac{\varepsilon}{\gamma_\varepsilon}\Big)^{\alpha\rho} 
\sum_{k=1}^{n_\varepsilon} 
\Big(1- \frac{(1-C)}{2}\Big(\frac{\varepsilon}{\gamma_\varepsilon}\Big)^{\alpha\rho}\Big)^{k-1} \frac{1}{k^{\alpha\rho}}\\
&\qquad + C \e^2 
\sum_{k=2}^{n_\varepsilon} \Big(1- \frac{(1-C)}{2}\Big(\frac{\varepsilon}{\gamma_\varepsilon}\Big)^{\alpha\rho}\Big)^{k-2} \frac{1}{k^{\alpha\rho}}.
\end{align*}
We identify 
\begin{align*}
\frac{1}{2}\Big(\frac{\varepsilon}{\gamma_\varepsilon}\Big)^{\alpha\rho} 
\sum_{k=1}^{n_\varepsilon} \Big(1- \frac{(1-C)}{2}\Big(\frac{\varepsilon}{\gamma_\varepsilon}\Big)^{\alpha\rho}\Big)^{k-1} 
\frac{1}{k^{\alpha\rho}} 
\lesssim_\e \e^{\kappa} \Li_{\al\rho}\Big(1-\frac{(1-C)}{2}\e^\kappa \Big),
\end{align*}
where 
\[
\kappa = \al\rho (1+ \frac{\al\rho}{\Gamma(1-\beta)}) 
\]
and $\Li_a(x) = \sum_{k=1}^\infty \frac{x^k}{k^a}$ 
is the polylogarithm function with parameter $a\in \RR$ and $x\in (0,1)$, 
a well-known analytic extension of the logarithm. 
Recall that $\al \rho < \frac{\al}{1+\al}<1$ due to (\ref{eq: al times rho}). 
By the following representation \cite{NIST}, Section 25.12, 
for $a\neq \NN$ and $0<x< 1$, 
given by 
\begin{align}
\Li_a(x) = \Gamma(1-a) (\ln(\frac{1}{x}))^{a-1} + \sum_{n=0}^\infty \zeta(a-n) \frac{(\ln(x))^{n}}{n!},
\end{align}
we obtain that for $a\in (0,1)$ 
\[
\lim_{x\nearrow 1} \Li_a(x)/(1-x)^{a-1} = \Gamma(1-a).
\]
Hence there is $C>0$ such that for $\e\in (0,\e_0)$ sufficiently small 
\begin{align*}
&\frac{1}{2}\Big(\frac{\varepsilon}{\gamma_\varepsilon}\Big)^{\alpha\rho} 
\sum_{k=1}^{n_\varepsilon} \Big(1- \frac{(1-C)}{2}\Big(\frac{\varepsilon}{\gamma_\varepsilon}\Big)^{\alpha\rho}\Big)^{k-1} \frac{1}{k^{\alpha\rho}} \\
&\lqq \e^{\kappa} \Li_{\al\rho}\Big(1-\frac{(1-C)}{2}\e^\kappa \Big)\\
&\lqq C \e^{\kappa} \e^{-\kappa (1-\al\rho)} = \e^{\kappa\al\rho} = S_{4}(\e) \searrow 0.
\end{align*}
The same polylogarithmic asymptotics is carried out for 
\begin{align*}
&C \e^2 
\sum_{k=2}^{n_\varepsilon} \Big(1- \frac{(1-C)}{2}\Big(\frac{\varepsilon}{\gamma_\varepsilon}\Big)^{\alpha\rho}\Big)^{k-2} \frac{1}{k^{\alpha\rho}}\\
&\lqq C \e^{2} \Li_{\al\rho}\Big(1-\frac{(1-C)}{2}\e^\kappa \Big)\\
&\lqq C \e^{2} \e^{-\kappa(1-\al\rho)} = \e^{2+\kappa\al\rho - \kappa} = S_5(\e) \searrow 0,
\end{align*}
since due to $\Gamma (1-\beta)<1$
\begin{align*}
2-(\frac{\al\rho}{\Gamma(1-\beta)} + 1)(\al\rho -1) \gqq 2 - (\al\rho -1)(\al\rho+1) = 2- (\al\rho^2 -1)= 3 -(\al\rho)^2 >0. 
\end{align*}

\paragraph{4) Estimate of the exit probabilities: } For any $m>0$ 
\begin{align*}
\sup_{y\in D_{3\delta_\e}^+} \mathbb{P}(\tau^{y, \varepsilon,-} \lqq m)
&= \sup_{y\in D_{3\delta_\e}^+} \mathbb{P}(e^{-\theta \varepsilon^\alpha \tau^{y, \varepsilon, -}} \gqq e^{\theta \e^{\alpha} m}) \\
&\lqq \sup_{y\in D_{3\delta_\e}^+} \mathbb{E} \Big[e^{-\theta \varepsilon^\alpha \tau^{y, \varepsilon, -}}\Big] e^{\theta \varepsilon^\alpha m} \\
&\lqq \underbrace{C(S_1(\varepsilon) + S_{2}(\varepsilon) + S_{3}(\varepsilon) + S_4(\varepsilon) + S_5(\varepsilon))}_{=: S(\varepsilon)} 
e^{\theta \varepsilon^\alpha m}.
\end{align*}
Replacing $m$ by $m_\varepsilon$ with $\limsup_{\varepsilon \rightarrow 0} m_\varepsilon \varepsilon^{\alpha} < \infty$ we obtain 
\begin{align*}
\sup_{y\in D_{3\delta_\e}^+} \mathbb{P}(\tau^{y, \varepsilon, -} \lqq m_\varepsilon) 
\lesssim_{\varepsilon}S(\varepsilon) \ra 0. 
\end{align*}
The function $S$ can be chosen to be a monotonic function. This finishes the proof. \\

\subsection{Consequences of the first exit result}

\begin{corollary}\label{cor: lower bound large scale}
Let the assumptions of the last theorem be satisfied and $\rho$ being chosen according to 
(\ref{eq: choice of rho}) and $\limsup_{\e\ra 0} m_\e \e^\al < \infty$. Construct recursively 
\begin{align*}
& U^{x, \e, 1}_t := \Big(u(t; x-\delta_\e)-\delta_\e + W_1 \ind\{t= T_1\}\Big) \wedge \gamma_\e, && t\in [0, T_1]\\
&U^{x, \e, n+1}_t := \Big(u(t-T_n; U^{x, \e, n}_{T_n}-\delta_\e)-\delta_\e + W_{n+1} \ind\{t = T_{n+1}-T_n\}\Big)\wedge \gamma_\e, 
&&t\in (0, T_{n+1}-T_n]\\
&Z^{x, \e}_t := \sum_{n=1}^\infty U^{x, \e, n}_t \ind\{t\in (T_n, T_{n+1}]\}, &&t\gqq 0.
\end{align*}
where the arrival times $T_n$ of the large jump increments $W_n$ are defined in (\ref{eq: Wi}), (\ref{eq: ti}) and (\ref{eq: la eps}). 
Then 
\begin{align*}
\liminf_{\e \ra 0} \inf_{x\gqq 3\delta_\e} \PP(\sup_{t\in [0, m_\e]} X^{x, \e}_t-Z^{x, \e}_t\gqq 0) = 1
\end{align*}
\end{corollary}

This is nothing but a reformulation of the proof of Theorem \ref{thm: exit}. 
The process we compare $X^{\e, x}$ to the deterministic solution $u(\cdot; x)$, 
starting in $x$ with large heavy-tailed jump increments $(T_n^\e, W_n^\e\wedge \gamma_\e)$, 
where the increments $W_n^\e$ are cut-off from below by a value $\gamma_\e$. 
The choice of $\gamma_\e$ has to satisfy two things: 
First, the deterministic trajectory has to overcome it during the waiting time 
$T_{n+1}^\e-T_n^\e$ with a probability tending to $1$. 
Second, for larger and larger initial values $i\gamma_\e < x\lqq (i+1)\gamma_\e$, 
the probability that $u(t,x) + \e W_i \lqq \gamma_\e$ has to 
decrease for growing $i$ and decreasing $\e$ with a sufficiently large. 

\begin{corollary}\label{cor: lower bound}
Let the assumptions of Theorem \ref{thm: exit} 
be satisfied and $\delta_\e$ being chosen according to (\ref{def: delta eps}). 
Then for any $m_\cdot: (0,1)\ra (0,\infty)$ satisfying $\lim_{\e\ra 0} m_\e \e^{\al\rho} =0$ 
we have 
\begin{align*}
\liminf_{\e \ra 0} \inf_{x\gqq 3\delta_\e} \PP(\sup_{t\in [0, m_\e]} X^{x, \e}_t-x^+_t \gqq -\delta_\e) = 1.
\end{align*}
\end{corollary}

\begin{proof} 
We keep the notation of the proof of Corollary \ref{cor: lower bound large scale}. 
First we obtain by a comparison argument that for all $x\gqq \delta_\e$ 
\[
u(t;x) \gqq x^+_t \qquad t\gqq 0. 
\]
Secondly we observe that $U^{x, \e, 1}_t = u(t;x)$ for $t< T_1$ 
and $\PP(T_1 \gqq m_\e) = e^{-m_\e \la_\e}  \approx_\e e^{-m_\e \e^{\al\rho}} \ra 1$. 
Hence combining these findings with inequality (\ref{ineq: y-u<0a}) we obtain  
\begin{align*}
\lim_{\e \ra 0} \inf_{x\gqq 3\delta_\e} \PP(\sup_{t\in [0, m_\e]} X^{x, \e}_t-x^+_t> -\delta_\e) = 0.
\end{align*}
\end{proof}

\begin{lemma}\label{lem: upper bound}
Let the assumptions of the Theorem \ref{thm: exit} 
be satisfied and $\delta_\e$ being chosen according to (\ref{def: delta eps}). 
Then there is a function $\ti m_\cdot: (0,1)\ra (0,\infty)$ satisfying 
$\lim_{\e\ra 0} \ti m_\e \delta_\e^{\beta} =0$ such that for any function $ \Delta_\cdot: (0, 1) \ra (0,1)$ 
wit $\lim_{\e\ra 0}\Delta_\e =0$ and 
and $\lim_{\e \ra 0} 3\delta_\e /\Delta_\e = 0$ we have 
\begin{align*}
\liminf_{\e \ra 0} \inf_{3\delta_\e\lqq x \lqq \Delta_\e} \PP(\sup_{t\in [0, \ti m_\e]} X^{x, \e}_t-x^+_t < \delta_\e^{\frac{\beta^2}{2}}\vee \Delta_\e^{1-\beta}) = 1.
\end{align*}
\end{lemma}

\begin{proof}
First choose $\rho$ we choose according to (\ref{eq: choice of rho}) in the proof of Theorem \ref{thm: exit} 
and fix for the moment $3\delta_\e \lqq x \lqq \Delta_\e$. 
Recall for $t\in [0, T_1]$ the notation 
\[
X^{\e, x}_t = Y^{\e, x}_t + \e W_1 \ind\{t = T_1\}
\]
and 
\[
V^{x, \e}_t = Y^{x, \e}_t- \e \xi^\e_t. 
\]
The subadditivity of $b(y) = B |y|^\beta$ on $(0, \infty)$ yields on the events 
$\{t< T_1\}$ and $\{\sup_{t\in [0, T_1]} |\e \xi^\e_s|\lqq \delta_\e\}$ that 
\begin{align*}
V^{\e, x}_t  &\lqq x+ \int_0^t b(V^{x, \e}_s) ds + B \delta_\e^\beta t 
~\lqq ~\Delta_\e + B \delta_\e^\beta \ti m_\e + \int_0^t b(V^{x, \e}_s) ds,
\end{align*}
where $\ti m_\e := \delta_{\e}^{-\frac{1}{2}\beta} \wedge r_\e$ 
with $\delta_\e = \e^{1-\rho(1+\al)} |\ln(\e)|^4$ in (\ref{def: delta eps}) 
and $r_\e = \e^{-\al\rho} |\ln(\e)|^2$ defined in (\ref{def: r eps}). 
Then Bihari's inequality \cite{M}, Theorem 8.3, implies 
\begin{align*}
&\sup_{t\in [0, \ti m_\e]} V^{\e, x}_t - x^+_t \\
&\lqq \sup_{t\in [0, \ti m_\e]} \bigg[\Big(\big(1-\beta) B t + (\Delta_\e+ B \delta_\e^\beta \ti m_\e\big)^{1-\beta}\Big)^\frac{1}{1-\beta} - ((1-\beta) Bt)^{\frac{1}{1-\beta}}\bigg]\\
&\lqq \bigg[\Big(\big(1-\beta) B \ti m_\e + (\Delta_\e+ B \delta_\e^\beta \ti m_\e)\big)^{1-\beta}\Big)^\frac{1}{1-\beta} 
- ((1-\beta) B\ti m_\e)^{\frac{1}{1-\beta}}\bigg]\\
&\lqq 2^{\frac{1}{1-\beta}-1} (\Delta_\e+ B \delta_\e^\beta \ti m_\e)^{1-\beta}\ra 0.
\end{align*}
Note that the bound of the right-hand side is of order 
\[
(\Delta_\e+ B \delta_\e^\beta \ti m_\e)^{1-\beta} \lesssim_\e \delta_\e^{\frac{\beta(1-\beta)}{2}} \vee \Delta_\e^{1-\beta}.
\]
We finish the proof by 
\begin{align*}
&\limsup_{\e \ra 0} \sup_{3\delta_\e\lqq x \lqq \Delta_\e} \PP(X^{\e, x}_t - x^+_t >\delta_\e^{\frac{\beta(1-\beta)}{2}} \vee \Delta_\e^{1-\beta})\\
&\lqq \limsup_{\e \ra 0} \Big(1- \PP(T_1> r_\e) - \PP(\sup_{t\in [0, r_\e]} |\e\xi^\e_t|>\delta_\e)\Big) = 0.
\end{align*}
\end{proof}

\noindent We obtain the main result of this section as a combination 
of Corollary \ref{cor: lower bound} and Lemma \ref{lem: upper bound}.

\begin{corollary}\label{cor: short time scale convergence}
Let the assumptions of the Theorem \ref{thm: exit} be satisfied and $\delta_\e$ chosen as in (\ref{def: delta eps}). 
Then for any $\Delta_\cdot: (0,1)\ra (0,1)$ monotonically increasing with $\lim_{\e \ra 0} \Delta_\e = 0$ and 
$\limsup_{\e \ra 0} 3\delta_\e / \Delta_\e \lqq 1$ there exists $\theta^*>0$ such that 
\[
\lim_{\e \ra 0+} \sup_{3\delta_\e\lqq x\lqq \Delta_\e} 
\PP(\sup_{t\in [0, \e^{-\theta^*}]} |X^{\e, x}_t - x^+_t| > \delta_\e^{\frac{\beta(1-\beta)}{2}}\vee \Delta_\e^{1-\beta}) = 0. 
\]
\end{corollary}

\section{The solution leaves a small environment of the origin in a short time} \label{sec: close to the origin}

Let us denote by $(X_t)_{t\gqq 0}$ the strong solution $(X^{\e, 0}_t)_{t\gqq 0}$ of system (\ref{SDE})
with initial value $x=0$.  
In addition we stipulate for $r_1, r_2 >0$ 
\begin{align}\label{eq: null austritt}
&\tau^{\e}(r_1, r_2):=\inf\left\{  t>0: X_t \lqq -r_1 \mbox{ or } X_t \gqq r_2\right\}.
\end{align}
and abbreviate for convenience $\tau_{r_1, r_2} = \tau^{\e}(r_1, r_2)$.

\subsection{When the noise strength meets the non-linear impact: space-time transition points}

\begin{proposition}\label{Prop noise 2} 
For 
\[\al > 1- (\beta^+\wedge \beta^-)\] 
and any $\vt\in (0,1]$ there is a family of monotonically increasing functions $\Theta^+_{\cdot, \vt}, \Theta^-_{\cdot, \vt}, t_{\cdot, \vt}: (0, 1) \ra (0,1)$ with 
$\lim_{\e \ra 0+} \Theta^+_{\e, \vt} =\lim_{\e \ra 0+} \Theta^-_{\e, \vt} = \lim_{\e \ra 0+} t_{\e, \vt} = 0$, such that  
for any function $\ho t_{\cdot, \vt}: (0,1) \ra (0, \infty)$ satisfying $\lim_{\e\ra 0}\hat{t}_{\e, \vt}/t_{\e, \vt}=+\infty$ 
we have
\[
\lim_{\e\ra 0} \PP\left(\tau_{\Theta^-_{\e, \vt}, \Theta^+_{\e, \vt}}>\ho{t}_{\e, \vt}\right)  =0.
\]
\end{proposition}

We omit the iteration argument by Markov property. The key result is the following. 

\begin{lemma}\label{lem: Hilfslemma 1}
Under the previous assumptions and 
\[\al > 1- (\beta^+\wedge \beta^-)\] 
and $\vt \in (0, 1]$ we have the following statement. 
There is a family of monotonically increasing functions $\Theta^+_{\cdot, \vt}, \Theta^-_{\cdot, \vt}, t_{\cdot, \vt}: (0, 1) \ra (0,1)$ with 
$\lim_{\e \ra 0+} \Theta^+_{\e, \vt} =\lim_{\e \ra 0+} \Theta^-_{\e, \vt} = \lim_{\e \ra 0+} t_{\e, \vt} = 0$, such that
we have 
\[
\lim_{\e \ra 0+} \PP\left(\tau_{\Theta^-_{\e, \vt}, \Theta^+_{\e, \vt}} > t_{\e, \vt}\right) < 1, 
\]
\end{lemma}

\noindent For notational convenience we will immediately the dependence on $\vt$, whenever possible. 

\begin{remark}
The parameter $\vt\in (0, 1]$ is a purely technical device, it will turn out in the next section 
that if $\beta^+ = \beta^-$ it cannot be chosen to be $1$ but only arbitrarily close to $1$. 
In any other case it will be set equal $1$. 
\end{remark}

\begin{proof}
For the convenience of notation we will fix $\vt \in (0, 1]$ and drop the respective subscript in the sequel. 
Assume there are $\Theta^+_\e, \Theta^-_\e, t_\e$ as in the statement of the lemma and let us 
abbreviate for convenience $\chi = \tau_{\Theta^-_{\e}, \Theta^+_\e}$. 
The definition of the event $\{\chi > t_{\e}\}$ implies
\[
- \Theta_\e^- \lqq  X_{t} \lqq \Theta_{\e}^+ \qquad \forall t\in [0, t_\e].
\]
Therefore, we infer from the event $\{\chi > t_{\e}\}$ for $t\in\left[0,t_{\e}\right]$ 
that 
\begin{align*}
\e L_{t}  & = X_t  -\int_0^t b(X_s) ds\\
&\lqq X_t +B^-\int_0^t (X_s)^{\beta^-} ds \\
&\lqq \Theta_\e^+ +B^- t_\e (\Theta_\e^-)^{\beta^-}.
\end{align*}
Analogously we obtain 
\begin{align*}
\e L_t &\gqq - \Theta_\e^- -B^+ t_\e (\Theta_\e^+)^{\beta^+}.
\end{align*}
If we now impose that the nonlinear term is asymptotically smaller, 
that is for instance $\Theta_\e^\beta t_\e^{1-\vt}$
than the boundary $\Theta_\e$
\begin{align}
B^+ t_{\e}(\Theta_{\e}^+)^{\beta^+}=\Theta_{\e}^- t_{\e}^{1-\vt}\nonumber\\
B^- t_{\e}(\Theta_{\e}^-)^{\beta^-}=\Theta_{\e}^+ t_\e^{1-\vt} \label{eq: boundaries equations part 1}
\end{align}
it follows 
\[
-(1+t_{\e}^{1-\vt}) \Theta_\e^-\lqq  \e L_t \lqq (1+t_{\e}^{1-\vt})\Theta_{\e}^+, \qquad t\in [0, t_\e], 
\]
and in particular $- (1+t_{\e}^{1-\vt}) \Theta_\e^-\lqq  \e L_{t_\e} \lqq (1+t_{\e}^{1-\vt}) \Theta_{\e}^+$.
As a first case we may assume that $\Theta_\e^+ / \Theta_\e^- \ra 0$ as $\e\ra 0$. 
If we stipulate for $\vt \in(0,1)$ 
\begin{equation}
\label{def: dritte Gleichung 1}\Theta^\circ_\e =\frac{\e t_\e^{\frac{1}{\al}}}{1+t_\e^{1-\vt}}
\end{equation} 
this yields 
\begin{align*}
&\PP\Big( -(1+t_\e^{1-\vt}) \Theta_{\e}^- \lqq \e L_{ t_{\e}} \lqq (1+t_\e^{1-\vt}) \Theta_{\e}^+ \Big) \\
&\PP\Big( -(1+t_\e^{1-\vt}) \frac{\Theta_{\e}^-}{\e t_\e^\frac{1}{\al}}  \lqq t_\e^{-\frac{1}{\al}}  L_{ t_{\e}} 
\lqq (1+t_\e^{1-\vt})\frac{\Theta_{\e}^+}{\e t_\e^\frac{1}{\al}} \Big) \\
& =\PP\Big( -(1+t_\e^{1-\vt}) \frac{\Theta_\e^-}{\Theta_\e^+}\frac{\Theta_{\e}^+}{\e t_\e^\frac{1}{\al}} 
\lqq L_{1}
\lqq (1+t_\e^{1-\vt}) \frac{\Theta_{\e}^+}{\e t_{\e}^{1/\al}} \Big)\\
&= \PP\Big( -\frac{\Theta_\e^-}{\Theta_\e^+} 
\lqq L_{1} \lqq 1 \Big)
 \stackrel{\e \ra 0}{\lra} \PP\Big(-\infty < L_1 \lqq 1\Big)<1. 
\end{align*}
As long as $\lim_{\e \ra 0} t_\e = 0$. The proof concludes 
with the following calculation which shows that for any exponent $\al \in (0,2)$, any powers $\beta^+, \beta^- \in (0,1)$ 
satisfying $\al \gqq 1-(\beta^+\wedge \beta^-)$ and $\e\in (0,1)$ 
the system (\ref{eq: boundaries equations part 1}) together 
either with (\ref{def: dritte Gleichung 1})  
as a unique solution $(\Theta^+_{\e, \vt}, \Theta^-_{\e, \vt}, t_{\e, \vt})_{\e, \vt \in (0,1]}$ 
with $\lim_{\e \ra 0+} t_{\e, \vt} = 0$ for any $\vt \in (0,1)$. 

We solve the equations for $t_\e$, $\Theta^+_\e$ and $\Theta^-_\e$ and 
start with the system (\ref{eq: boundaries equations part 1}) which implies by reinsertion 
\begin{align*}
\Theta^-_\e &= t_\e^{\vt} B^+(\Theta_\e^+)^{\beta^+} \\
&= t_\e^{\vt} B^+(t_\e^{\vt} B^-(\Theta_\e^-)^{\beta^-})^{\beta^+}\\
&= B^+(B^-)^{\beta^+} t_\e^{\vt(1 + \beta^+)}  (\Theta_\e^-)^{\beta^+\beta^-},
\end{align*}
and 
\begin{align*}
&(\Theta^-_\e)^{1-\beta^+\beta^-} =B^+(B^-)^{\beta^+} t_\e^{\vt(1+\beta^+)} \\
\vzv &\qquad \Theta^-_\e = \Big(B^+(B^-)^{\beta^+} t_\e^{\vt(1+\beta^+)} \Big)^\frac{1}{1-\beta^+\beta^-}
= (B^+)^\frac{1}{1-\beta^+\beta^-}\; (B^-)^\frac{\beta^+}{1-\beta^+\beta^-}\; t_\e^\frac{\vt(1+\beta^+)}{1-\beta^+\beta^-}
\end{align*}
and by symmetry 
\begin{align*}
&(\Theta^+_\e)^{1-\beta^+\beta^-} =B^-(B^+)^{\beta^-} t_\e^{\vt(1+\beta^-)} \\
\vzv &\qquad \Theta^+_\e = \Big(B^-(B^+)^{\beta^-} t_\e^{\vt(1+\beta^-)} \Big)^\frac{1}{1-\beta^+\beta^-}
= (B^-)^\frac{1}{1-\beta^+\beta^-}\; (B^+)^{\frac{\beta^-}{1-\beta^+\beta^-}}\; t_\e^\frac{\vt(1+\beta^-)}{1-\beta^+\beta^-}.
\end{align*}
Denote by $\beta^\circ := \beta^+ \wedge \beta^-$ and $\beta^* := \beta^+ \vee \beta^-$. 
The last two formulas yield 
\begin{align*}
&\Theta_\e^* := 
\Theta_\e^+ \vee \Theta_\e^- 
= (B^\circ)^\frac{1}{1-\beta^\circ\beta^*}\; (B^*)^\frac{\beta^\circ}{1-\beta^\circ\beta^*}\; 
t_\e^\frac{\vt(1+\beta^\circ	)}{1-\beta^\circ\beta^*}\\
&\Theta_\e^\circ := 
\Theta_\e^+ \wedge \Theta_\e^- 
= (B^*)^\frac{1}{1-\beta^\circ\beta^*}\; (B^\circ)^\frac{\beta^*}{1-\beta^\circ\beta^*}\; 
t_\e^\frac{\vt(1+\beta^*)}{1-\beta^\circ\beta^*}
\end{align*}
As a consequence, we obtain for $\beta^\circ < \beta^*$ 
\begin{align}
&\lim_{\e\ra 0+} \Theta_\e^\circ/ \Theta_\e^* = 0.\label{eq: order of the boundaries}
\end{align}
and for $\beta = \beta^* = \beta^\circ$ 
\begin{align}
&\frac{\Theta_\e^\circ}{\Theta_\e^*} = 
\frac{(B^*)^\frac{1}{1-\beta^2}\; (B^\circ)^\frac{\beta}{1-\beta^2}}
{(B^\circ)^\frac{1}{1-\beta^2}\; (B^*)^\frac{\beta}{1-\beta^2}}
= \Big(\frac{B^\circ}{B^*}\Big)^{-\frac{1}{1+\beta}}.
\label{eq: order of the boundaries equal }
\end{align}
We complement the system (\ref{eq: boundaries equations part 1}) 
by equation (\ref{def: dritte Gleichung 1}). 
Plugging in directly yields  
\begin{align*}
\vzv &\qquad \e = (B^*)^\frac{1}{1-\beta^*\beta^\circ}\; (B^\circ)^\frac{\beta^*}{1-\beta^*\beta^\circ}\; 
t_\e^{\frac{\vt(1+\beta^*)}{1-\beta^*\beta^\circ}-\frac{1}{\al}}.
\end{align*}
We examine the exponent
\begin{align*}
\frac{\vt(1+\beta^*)}{1-\beta^*\beta^\circ}-\frac{1}{\al} 
&= \frac{\vt\al(1+\beta^*)-1+\beta^*\beta^\circ}{\al(1-\beta^*\beta^\circ)}
&= \frac{\vt \al -1 +\beta^*(\vt \al +\beta^\circ)}{\al(1-\beta^*\beta^\circ)}
\gqq \frac{\vt \al -1 +\beta^*}{\al(1-\beta^*\beta^\circ)}>0,
\end{align*}
since $\vt \al + \beta^\circ>1$ and therefore $\vt \al + \beta^* >1$ 
we have 
\begin{align}
&\qquad \e = (B^*)^\frac{1}{1-\beta^*\beta^\circ}\; (B^\circ)^\frac{\beta^*}{1-\beta^*\beta^\circ}\; 
t_\e^{\frac{\vt \al +\beta^*- 1 + \beta^* (\vt \al +\beta^\circ-1)}{\al(1-\beta^*\beta^\circ)}}\nonumber\\
\vzv &\qquad 
t_\e = \frac{\e^\frac{\al(1-\beta^*\beta^\circ)}{\vt \al +\beta^*- 1 + \beta^* (\vt \al +\beta^\circ-1)}}
{(B^\circ)^\frac{\al \beta^*}{\al +\beta^*-1 + \beta^*(\al +\beta^\circ-1)}\; (B^*)^\frac{\al}{\al +\beta^*-1 + \beta^*(\al +\beta^\circ-1)}}
. \label{eq: t epsilon >1}
\end{align}
We obtain 
\begin{align*}
\Theta^+_\e
&= (B^-)^\frac{1}{1-\beta^\circ\beta^*}\; (B^+)^\frac{\beta^-}{1-\beta^\circ\beta^*}\; t_\e^\frac{\vt(1+\beta^-)}{1-\beta^\circ\beta^*} \\
&= \frac{(B^-)^\frac{1}{1-\beta^\circ\beta^*}\; (B^+)^\frac{\beta^-}{1-\beta^\circ\beta^*}}{(B^\circ)
^\frac{\al\beta^*}{\al +\beta^*-1 + \beta^*(\al +\beta^\circ-1)}\; 
(B^*)^\frac{\al}{\al +\beta^*-1 + \beta^*(\al +\beta^\circ-1)}}\; 
\e^{\frac{\vt\al(1+\beta^-)}{\vt \al +\beta^*- 1 + \beta^* (\vt \al +\beta^\circ-1)}}
\end{align*}
and
\begin{align*}
\Theta^-_\e 
&= \frac{(B^+)^\frac{1}{1-\beta^\circ\beta^*}\; (B^-)^\frac{\beta^+}{1-\beta^\circ\beta^*}}{(B^\circ)
^\frac{\al\beta^*}{\al +\beta^*-1 + \beta^*(\al +\beta^\circ-1)}\; 
(B^*)^\frac{\al}{\al +\beta^*-1 + \beta^*(\al +\beta^\circ-1)}}\; 
\e^{\frac{\vt\al(1+\beta^+)}{\vt\al +\beta^*-1 + \beta^*(\vt\al +\beta^\circ-1)}}.
\end{align*}
These calculations establish the existence and uniqueness of $(\Theta^+_{\e, \vt}, \Theta^-_{\e, \vt}, t_{\e, \vt})_{\e, \vt \in (0,1]}$ 
as claimed in the statement of Lemma \ref{lem: Hilfslemma 1}.
\end{proof}

\subsection{The exit locations from a neighborhood of the origin}\label{sec: exit distributions}

For a fixed parameter $\vt$ fixed and we denote by $\chi :=  \chi_\e := \tau_{\Theta^+_\e, \Theta^-_\e}$ as defined 
in (\ref{eq: null austritt}) 
and $(\Theta^+_\e, \Theta^-_\e, t_\e)_{\e \in (0,1]}$ 
defined by Definition~\ref{def: Theta-t} and Lemma~\ref{lem: Hilfslemma 1}. 
In this subsection we determine the asymptotic probabilities  
\[
\PP(X^\e_{\chi} \gqq \Theta^+_\e)\quad  \mbox{ and } \quad\PP(X^\e_{\chi} \lqq -\Theta^-_\e)
\]
in the limit of small $\e$.

\begin{proposition}\label{prop: exit location}
For $\al> 1- (\beta^* \wedge \beta^\circ)$ and $\beta^+\neq \beta^-$ we have 
\begin{equation}
\PP(X^\e_{\chi} \gqq \Theta^+_\e)
= 
\begin{cases} 1 & \beta^+ < \beta^-,\\ 0 & \beta^+ >\beta^-.\end{cases}
\end{equation}
\end{proposition}

\subsubsection{Close to the transition points the non-linear impact of the noise remains subcritical}

This section controls that there is no exit by non-linear impact of the noise. 
We decompose $X^\e$ into the sum of $V^\e$ and $\e L$, where $V_t^\e := X^\e_t - \e L_t$. 
It satisfies 
\begin{align*}
V^\e_t = \int_0^t b(V^\e_s + \e L_s) ds, \qquad t\gqq 0. 
\end{align*}

\begin{lemma}\label{lem: reduction process to noise exit asymmetric case}
Assume $\beta^+> \beta^-$ the parametrized family of functions $(\Theta^+_{\e, 1}, \Theta^-_{\e,1}, t_{\e,1})_{\e \in (0,1]}$ 
determined in Definition \ref{def: Theta-t}. 
Then there exists $g>0$ such that for $\hat t_\e := t_\e |\ln(\e)|, \e \in (0,1)$ we have 
\begin{align*}
\PP(\sup_{t\in [0, \hat t_\e]} (V^\e_t)_+ > \Theta^+_\e \e^g) \ra 0. 
\end{align*}
\end{lemma}

\begin{proof}
As in the previous lemma the self-similarity 
\begin{align*}
\sup_{t\in [0, \hat t_\e]} (\e L_t)_+^{\beta^*} \stackrel{d}{=} \e^{\beta^*} \hat t_{\e}^{\frac{\beta^*}{\al}} (L_1)_+^{\beta^*} 
 \end{align*}
yields 
\begin{align*}
\PP(\sup_{t\in [0, \hat t_\e]} (\e L_t)_+^{\beta^*} > \Theta^*_\e \e^g) 
&\lqq \PP(\sup_{t\in [0, \hat t_\e]}\e^{\beta^*} \hat t_{\e}^{\frac{\beta^*}{\al}} (L_1)_+^{\beta^*}> \Theta^*_\e \e^g)\\
\end{align*}
We check whether 
\begin{align*}
\e^{\beta^*} t_\e^{\frac{\al + \beta^*}{\al} - \frac{1+\beta^\circ}{1-\beta^*\beta^\circ}} \ra 0, \mbox{ as } \e \ra 0. 
\end{align*}
Check the exponent 
\begin{align}
&\beta^* + \frac{\al (1-\beta^\circ\beta^*)}{\al + \beta^*-1 + \beta^*(\al +\beta^\circ -1)} 
\Big(\frac{\al + \beta^*}{\al} - \frac{1+\beta^\circ}{1-\beta^*\beta^\circ}\Big)\nonumber\\
&= \frac{\beta^* (\al + \beta^*-1 + \beta^*(\al +\beta^\circ -1))+ (\al + \beta^*)(1-\beta^*\beta^\circ) 
- \al (1+\beta^\circ)}{\al + \beta^*-1 + \beta^*(\al +\beta^\circ -1)}\label{def: exponent}
\end{align}
By assumption the denominator is positive. The enumerator behaves as 
\begin{align*}
&\beta^* (\al + \beta^*-1 + \beta^*(\al +\beta^\circ -1))+(\al+ \beta^*)(1-\beta^*\beta^\circ) - \al (1+\beta^\circ)\\
& = \al \beta^* + (\beta^*)^2 - \beta^* + \al (\beta^*)^2 + \beta^\circ (\beta^*)^2 
 - (\beta^*)^2 + \al + \beta^* -\al \beta^\circ \beta^*- \beta^\circ (\beta^*)^2 -\al - \al\beta^\circ\\
& = \al \beta^*  + \al (\beta^*)^2  -\al \beta^\circ \beta^* - \al \beta^\circ\\
& = \al (\beta^* - \beta^\circ) + \al \beta^*(\beta^*-\beta^\circ)>0.
\end{align*}
We set $2g$ equal to the expression in (\ref{def: exponent}). 
\end{proof}

\begin{definition}\label{def: Theta-t}
Let $\alpha \in (0,2)$ and $\beta^+, \beta^- \in (0,1)$ given satisfying $\alpha > 1- (\beta^+\wedge \beta^-)$. 
For any $\alpha$-stable noise $L$, 
we define the family $(\Theta_{\e}^+, \Theta^-_{\e}, t_{\e})_{\e\in (0,1)} := (\Theta_{\e, \vt^*}^+, \Theta^-_{\e, \vt^*}, t_{\e, \vt^*})_{\e\in (0,1)}$ 
defined in the proof of Lemma \ref{lem: Hilfslemma 1} with 
\begin{equation}
\vt^* = 
\begin{cases} 
\frac{1}{2}(1+\frac{1-\beta}{\al}) &\mbox{ if }\beta^* = \beta^\circ,\\
1  &\mbox{ else. }
\end{cases}
\end{equation}
\end{definition}

\subsubsection{The spatial exit probabilities from the space-time box of the transition points}

In the sequel we determine $\lim_{\e \ra 0+} \PP(\e L_\chi\gqq \Theta_\e^+)$. 
This exit problem will be mainly treated in the spirit of the Brownian case 
as for instance in the book of Revuz and Yor \cite{revuz-yor}. 
Denote for $\kappa \in \RR$ and $\e\in (0,1)$ the jump time 
\[
\tau_\kappa(\e) := \inf\{t>0~|~|\Delta L_t|>\e^{-\kappa}\}. 
\]
The appropriate choice of $\kappa \in \RR$ allows to give an estimate for 
the first exit problem of $\e L$ from $[-\Theta^-_\e, \Theta^+_\e]$ 
in the sense of Revuz and Yor, since $|\e \Delta \xi^\kappa_t|\lqq C \e^{1-\kappa}\lqq \Theta^*_\e$. 
This means the jump to exit the interval is small in comparison to the boundary and vanishes in the limit of small $\e$. 
$\kappa$ should verify two propoerties. 
First it has to ensure that jumps beyond the threshold $\e^\kappa$ occur after $t_\e$,
with a probability mass which tends to $1$. More precisely, since 
\[
\nu(B^c_{\e^{-\kappa}}(0)) = 2 \int_{\e^{-\kappa}}^\infty \frac{dy}{y^{\al+1}} 
= \frac{-2}{\al} y^{-\al} \Big|_{\e^{-\kappa}}^\infty = \frac{2}{\al} \e^{\kappa \al}
\]
we impose on $\kappa$ that 
\begin{equation}\label{eq: tau kappa}
\PP(\tau_\kappa(\e) > t_\e) = \exp(-\frac{2}{\al} \e^{\kappa \al} t_\e)\ra 1, \quad \mbox{ as }\e \ra 0. 
\end{equation}
This is satisfied if 
\begin{equation}\label{eq: time condition on kappa}
\kappa \al > \frac{-\al (1-\beta^\circ \beta^*)}{\vt\al + \beta^*-1 + \beta^*(\vt\al + \beta^\circ -1)}.  
\end{equation}
As a second crucial feature we need 
\begin{equation}\label{eq: tau kappa 2}
\e^{1-\kappa} / \Theta^*_\e  
\ra 0, \qquad \mbox{ as } \e \ra 0+.
\end{equation}
This imposes
\begin{equation}\label{def: kappa} 
1- \kappa > \frac{\vt \al (1+ \beta^\circ)}{\vt\al + \beta^* -1 + \beta^*(\vt \al + \beta^\circ -1)}.
\end{equation} 
We verify that the conditions (\ref{eq: time condition on kappa}) and (\ref{eq: tau kappa 2}) reading 
\begin{equation}\label{eq: choice of cut-off}
\frac{\vt \al (1+ \beta^\circ)}{\vt\al + \beta^* -1 + \beta^*(\vt \al + \beta^\circ -1)} -1 < - \kappa < \frac{(1-\beta^\circ \beta^*)}{\vt\al + \beta^*-1 + \beta^*(\vt\al + \beta^\circ -1)}
\end{equation}
can be satisfied simultaneously. 
On the common the denominator we have to verify the positivity of the enumerators' difference 
\begin{align*}
&1-\beta^\circ \beta^* - \vt \al (1+ \beta^\circ) +(\vt\al + \beta^* -1 + \beta^*(\vt \al + \beta^\circ -1)) \\
& = -\beta^\circ \beta^* - \vt \al - \vt \al \beta^\circ +  \vt\al + \beta^*  + \beta^*\vt \al + \beta^*\beta^\circ -\beta^* \\
& =   \vt \al (\beta^* -\beta^\circ) > 0. 
\end{align*}
For $\vt\in (0,1)$ we may fix 
\begin{equation}\label{def: kappa precise}
\kappa = - \frac{(1-\beta^\circ \beta^*)-\vt^2 \al (\beta^* -\beta^\circ)}{\vt\al + \beta^*-1 + \beta^*(\vt\al + \beta^\circ -1)} 
\end{equation}
Going back to the Poisson random measure representation of $L$ 
we then have for any $T> 0$ 
\begin{align}
L_t &= \int_0^t \int_{|y|\lqq 1} y (N(ds dy) - ds \nu(dy)) + \int_{0}^t \int_{|y|>1} y N(ds dy) \nonumber\\
&= \int_0^t \int_{|y| \lqq \e^{-\kappa} } y \ti N(ds dy) 
\qquad \PP(~\cdot~|~\tau_\kappa(\e)\gqq T)-\mbox{a.s. for all }t\in [0, T]. \label{eq: L conditioned on small jumps}
\end{align}
The first summand is given as the L\'evy martingale $(\xi^\kappa_t)_{t\gqq 0}$, 
\begin{align*}
\xi^\kappa_t = \int_0^t \int_{|y| \lqq \e^{-\kappa} } y \ti N(ds dy), \qquad t\gqq 0. 
\end{align*}

We define the for $r^+, r^->0$ and $\e>0$ the hitting times 
of $\RR \setminus (-\Theta^-_\e, \Theta^+_\e)$ 
\begin{align}
\si^+_{r^+} &:= \inf\{t>0~|~\e \xi^\kappa_t \gqq r^+\},\nonumber\\
\si^-_{r^-} &:= \inf\{t>0~|~\e \xi^\kappa_t \lqq - r^-\},\nonumber\\
\si_{r^+, r^-} &:= \si^+_{r^+} \wedge \si^-_{r^-}.\label{def: sigma}
\end{align}

\begin{lemma}\label{lem: noise exit strictly stable} 
Under these assumptions and $\beta^+>\beta^-$ we obtain 
\begin{align*}
\limsup_{\e\ra 0} \PP(\si^+_{\Theta^+_\e}< \si^-_{\Theta^-_\e}) 
&\lqq \limsup_{\e\ra 0} \frac{\Theta^-_\e+ \e^{1-\kappa}}{\Theta^+_\e+ \Theta^-_\e+ \e^{1-\kappa}} = 0. 
\end{align*}
\end{lemma}

\begin{proof}
The definition of the exit times and the choice of the jump size yields the estimates 
\begin{align*}
&\e \xi^\kappa_{\si^-} \gqq -(r^-+\e^{1-\kappa}) \quad \mbox{ and }\quad \e \xi^\kappa_{\si^{-}-} < -r^-
\qquad \mbox{ a.s.  on the event }\{\si^{-}\lqq n\}.
\end{align*}
For $r_1, r_2>0$ and $n\in \NN$ given we fix
\begin{align*}
\si^{+, n} &:= \si^+_{r^+} \wedge n, \\
\si^{-, n} &:= \si^-_{r^-} \wedge n, \\
\si^{n} &:= \si^{+, n} \wedge \si^{-, n}.
\end{align*}
Applying the optional stopping theorem we obtain 
\begin{align*}
0 &= \EE[\e \xi^\kappa_{\si^n}] \\
&= \EE[\e \xi^\kappa_{\si^n} (\ind\{\si^{+, n} < \si^{-, n}\}+ \ind\{\si^{+, n}\gqq \si^{-,n}\}) ]\\
&= \EE[\e \xi^\kappa_{\si^{+,n}} \ind\{\si^{+,n}< \si^{-,n}\}+ \e \xi^\kappa_{\si^{-,n}} \ind\{\si^{+,n}\gqq \si^{-,n}\}]
\end{align*}
we may estimate 
\begin{align*}
0 &= \EE[\e \xi^\kappa_{\si^{+,n}} \ind\{\si^{+,n}< \si^{-,n}\}+ \e \xi^\kappa_{\si^{-,n}} \ind\{\si^{+,n}\gqq \si^{-,n}\}]\\
&\gqq r^+ \PP(\{\si^+_{r^+}< \si^-_{r^-}\}\cap \{\si_\e \lqq n\}) - (r^-+\e^{1-\kappa}) \PP(\si^-_{r^-}\lqq \si^+_{r^+}). 
\end{align*}
Letting $n$ tend to $\infty$ we obtain 
\begin{align}
0 &\gqq r^+\PP(\si^+_{r^+}< \si^-_{r^-}) - (r^-+\e^{1-\kappa}) (1-\PP(\si^+_{r^+}<\si^-_{r^-}))
\label{eq: proba estimate for abstract boundaries}
\end{align}
and eventually 
\begin{align}\label{eq: asymptotic proto prob}
\PP(\si^+_{\Theta^+_\e}< \si^-_{\Theta^-_\e}) 
\lesssim_\e \frac{r^- +\e^{1-\kappa}}{r^+ + r^- +\e^{1-\kappa}}.
\end{align}
The choice of $\kappa$ 
now entails that $r^+$ replaced by $\Theta^+_\e = \Theta^*_\e$ leads to 
\begin{align*}
\e^{1-\kappa} &\lesssim_\e \Theta^+_\e = C^+ \e^{\frac{\vt \al(1+\beta^-)}{\vt\al +\beta^*(\vt\al +\beta^\circ)-1}},
\end{align*}
and analogously for $r^-$ being replaced by $\Theta^-_\e$ eventually leading to the desired result
\begin{align}\label{eq: asymptotic prob}
\PP(\si^+_{\Theta^+_\e}< \si^-_{\Theta^-_\e}) 
\lesssim_\e \frac{\Theta^-_\e +\e^{1-\kappa}}{\Theta^+_\e + \Theta^-_\e +\e^{1-\kappa}} \ra  0, 
\qquad \mbox{ as }\e \ra 0.
\end{align}
\end{proof}

\begin{proof} of Proposition \ref{prop: exit location}: 

Without loss of genrality, due to $\beta^+>\beta^-$ we may collect Lemma \ref{lem: reduction process to noise exit asymmetric case}, 
Lemma \ref{lem: noise exit strictly stable}, 
equation (\ref{eq: tau kappa}) and Proposition \ref{Prop noise 2}, 
which altogheter guarantee the existence of a constant $g>0$ such that for $\hat t_\e = t_\e |\ln(\e)|$ such that 
\begin{align*}
\PP(X^\e_\chi \gqq \Theta^+_\e) 
&\lqq \PP(X^\e_\chi \gqq \Theta^+_\e, \chi \lqq \hat t_\e) + \PP(\chi > \hat t_\e) \\
&\lqq \PP(\sup_{t\in [0, \hat t_\e]} (V_t^\e)_+^{\beta^+} + \e L_\chi\gqq \Theta_\e^+) + \PP(\chi > \hat t_\e) \\
&\lqq \PP(\sup_{t\in [0, \hat t_\e]} (V_t^\e)_+^{\beta^+}\gqq \Theta_\e^+\e^g) + \PP(\e L_\chi \gqq \Theta^+_\e(1-\e^g))+ \PP(\chi > \hat t_\e)\\
&\lqq \PP(\sup_{t\in [0, \hat t_\e]} (V_t^\e)_+^{\beta^+}\gqq \Theta_\e^+\e^g) + \PP(\si^+_{\Theta^+_\e}< \si^-_{\Theta^-_\e}) 
+ \PP(\tau_\kappa \lqq \hat t_\e) + \PP(\chi > \hat t_\e) \ra 0,
\end{align*}
as $\e \ra 0+$. 
Eventually the relation 
$\liminf_{\e \ra 0+} \PP(X^\e_\chi \lqq - \Theta^-_\e) \gqq 1- \limsup_{\e \ra 0+} \PP(X^\e_\chi \gqq \Theta^+_\e)$ 
finishes the proof.

\end{proof}

\newpage
\section{The linearized dynamics enhances the regime close to the origin}\label{sec: linearized}

We already know by Section \ref{sec: large jumps}, Corollary \ref{cor: short time scale convergence} 
that for initial values $x\gqq -3\delta_\e$ the law 
$\PP \circ X^{\e, x} \ra \delta_{x^+}$ uniformly on larger and larger time scales. 
Section \ref{sec: exit distributions} establishes 
that for any 
family of functions $(\Theta^+_{\e}, \Theta^-_{\e}, t_{\e})_{\e\in (0,1]}$ defined by Definition \ref{def: Theta-t} 
and initial values $x\in (-\Theta^-_\e, \Theta^+_\e)$ the solution 
$X^{x, \e}$ exits the interval $(-\Theta^-_\e, \Theta^+_\e)$ in time $\ti t_\e$ almost surely as long as 
$\lim_{\e\ra 0} \ti t_\e / t_\e \ra 0$.
In order to fill the gap between the scales of initial values 
\[
\Theta^\pm_\e = \e^\frac{\al(1+\beta^\pm)}{\al+\beta^\circ -1 + \beta^*(\al+\beta^\circ -1)} \lesssim_\e 3\e^{1- \rho^\pm(1+\al)} = 3\delta_\e^\pm,
\]
we consider the linearized dynamics. 

The main result of this section tells us that with a probability tending to $1$, 
the solution exits on the outer boundary of $[-3\delta_\e^-, -\Theta^-_\e] \cup [\Theta^+_\e, 3\delta_\e^+]$. 
We treat each subinterval individually with out loss of generality $[\Theta^+_\e, 4\delta_\e^+]$. 

\begin{lemma}
Let $[t_0 ,t_1) \subset \RR$ und  $v,v_0, \phi : [t_0, t_0) \ra \RR$, 
where $v$  und  $v_0$  measurable and locally bounded functions and $\phi \in L^1([t_0, t_1), \RR)$ with $\phi\gqq 0$. 
Then for almost all $t \in [t_0, t_1)$
\begin{align*}
v(t) \gqq v_0(t) + \int_{t_0}^t \phi(s) v(s) ds 
\end{align*}
implies for almost all $t\in [t_0, t_1)$ 
\begin{align*}
v(t) \gqq v_0(t) + \int_{t_0}^t v_0(s) \phi(s) \exp(\int_s^t \phi(r) dr) ds.  
\end{align*}
\end{lemma}

\noindent For $\e >0$ and $x \in [\Psi_{0}, \Psi_1]$ denote
\[
\upsilon^{x}(\e) := \inf\{t>0~|~ X^{\e, x}_t \gqq \Psi_1\}.
\]

\begin{proposition}\label{prop: linearized dynamics}
Without loss of generality we consider $\beta = \beta^\circ = \beta^+ \lqq \beta^-$ and 
the family of functions $(\Theta^+_{\e}, \Theta^-_{\e}, t_{\e})_{\e\in (0,1]}$ defined by Definition \ref{def: Theta-t}. 
Then there is an increasing, continuous function $s^n_\cdot: (0,1) \ra (0,1)$ with $s^n_\e \ra 0$ for any fixed $n \in \NN$ 
as $\e\ra 0$, such that 
\[
\lim_{\e\ra 0}\sup_{x\gqq \Psi_{0, \e}} \PP(\upsilon^{x}(\e) > s_{\e}) = 0.  
\]
\end{proposition}

\begin{proof}
We start with setting $\Psi_{0, \e} = \Theta^\circ_\e$ and 
introduce the time $s_{\e}$ with $s_{\e} \ra 0$, as $\e\ra 0$, and $\Psi_{1, \e} \gqq \Psi_{0,\e}$, with $\Psi_{1, \e} \ra 0$, 
as $\e \ra 0$, which both will be determined below. 
The proof consists in the establishment of an appropriate choice of a parameter $\pi_1 \in \RR$. 
For $\pi_1 \in \RR$ given we denote the time 
\[
\tau_{\pi_1} = \tau_{\pi_1}(\e) := \inf\{t>0~|~ |\Delta L_t|>\e^{-\pi_{1}}\}.
\]
We write shorthand $\beta, B$ for $\beta^+$, $B^+$.
Then analogously to (\ref{eq: L conditioned on small jumps}) we have 
\begin{align*}
\e L_t = \e \xi^{\pi_1}_t  \qquad \mbox{ for all }t \in [0, s_{\e}]\mbox{ and }\PP(\cdot~|~ \tau_{\pi_1} > s_{\e})-\mbox{ a.s.}
\end{align*}
Hence for all 
$\om \in \{\tau_{\pi_1} > s_{\e}\} \cap \{\sup_{t\in [0, s_{\e}]} |\e \xi^{\pi_1}_t| \lqq \frac{B}{2}\Psi_{0,\e}^{\beta} ~s_{\e}\}$ 
we have $\PP(\cdot~|~ \tau_{\pi_1} > s_{\e})-\mbox{ a.s.}$ for $t\in [0, s_{\e}]$ that 
\begin{align*}
X^{\e, x}_{t} &= x + \int_0^{t} b(X^{\e, x}_s) ds + \e L_{t} 
= x + \int_0^{t} b(X^{\e, x}_s) ds + \e \xi^\pi_t  \\
&\gqq \Psi_{0,\e} + B\int_0^{t} \Big[ \Psi_{0,\e}^\beta +  (X^{\e, x}_s - \Psi_{0, \e}) 
\frac{\Psi_{1,\e}^\beta-\Psi_{0, \e}^\beta}{\Psi_{1, \e}-\Psi_{0,\e}} \Big] ds + \e \xi^{\pi}_t\\
&\gqq \Psi_{0,\e} + B\int_0^{t} \Big[ \frac{\Psi_{0,\e}^\beta}{2} +  (X^{\e, x}_s - \Psi_{0, \e}) 
\frac{\Psi_{1,\e}^\beta-\Psi_{0, \e}^\beta}{\Psi_{1, \e}-\Psi_{0,\e}} \Big] ds.
\end{align*}
Hence for $W_{1,t} := W^{\e, x}_{1,t} := X^{\e, x}_t- \Psi_{0,\e}$ and $t\gqq 0$ we have 
\begin{align*}
W_t &\gqq \frac{B}{2}\int_0^{t} \Big[\Psi_{0,\e}^\beta +  W_s
\frac{\Psi_{1,\e}^\beta-\Psi_{0,\e}^\beta}{\Psi_{1,\e}-\Psi_{0,\e}}\Big] ds \\
& \gqq \frac{B}{2}\Psi_{0,\e}^\beta t + B\int_0^{t}W_s \Big[\frac{\Psi_{1,\e}^\beta-\Psi_{0,\e}^\beta}{\Psi_{1,\e}-\Psi_{0,\e}}\Big] ds\\
& \gtrsim_\e \frac{B}{2}\Psi_{0,\e}^\beta t + \frac{B}{2} \frac{1}{\Psi_{1,\e}^{1-\beta}}\int_0^{t} W_s ds.
\end{align*}
A classical non-autonomous Gronwall inequality from below yields  
\begin{align*}
W_{t} 
&\gqq \frac{B}{2}\Psi_{0,\e}^\beta t
 + \frac{B}{2} \Psi_{0,\e}^\beta \exp\Big(\frac{B}{2} \frac{t}{\Psi_{1,\e}^{1-\beta}}\Big) 
\int_0^{t} s \exp\Big(-\frac{B}{2} \frac{s}{\Psi_{1,\e}^{1-\beta}}\Big)ds
\end{align*}
and by direct calculation 
\begin{align*}
W_{t} 
&\gqq 
\frac{B}{2}\Psi_{0,\e}^\beta t 
+ \frac{2}{B}\Psi_{0,\e}^\beta \Psi_{1,\e}^{2(1-\beta)}  \exp\Big(\frac{B}{2} \frac{t}{\Psi_{1,\e}^{1-\beta}}\Big) 
 \Big(1 -(1+\frac{B}{2} \frac{t}{\Psi_{1,\e}^{1-\beta}}) \exp\big(-\frac{B}{2} \frac{t}{\Psi_{1,\e}^{1-\beta}} \big)\Big).
\end{align*} 
We set $s_{\e} = \frac{2}{B} \Psi_{1,_\e}^{\frac{1-\beta}{2}}$. 
This choice yields for any $C>0$ a constant $\e_0\in (0,1)$ such that $0< \e\lqq \e_0$ 
\[
(1+\frac{B}{2} \frac{s_{\e}}{\Psi_{1,\e}^{1-\beta}}) \exp\big(-\frac{B}{2} \frac{s_{\e}}{\Psi_{1,\e}^{1-\beta}} \big)
\lqq C.
\]
Therefore for $\e \in (0, \e_0]$ 
\begin{align*}
X^{\e, x}_{s_\e} &\gqq 
\Psi_{0,\e}
+\frac{B}{2}\Psi_{0, \e}^\beta s_{\e} \\
&\qquad + \Psi_{1,\e}^{2(1-\beta)} \Psi_{0,\e}^\beta \
\Big[\frac{2}{B}\Big(1 -C\Big) 
\Big]\exp\Big(\frac{B}{2} \frac{s_{\e}}{\Psi_{1,\e}^{1-\beta}}\Big)\\
&\gtrsim_\e \frac{1-C}{B} \Psi_{1,\e}^{2(1-\beta)} \Psi_{0,\e}^\beta \exp\Big(\Psi_{1,\e}^{-\frac{1}{2}(1-\beta)}\Big)
\gtrsim_\e \Psi_{1, \e}.
\end{align*}
We hence obtain 
\begin{align}
\PP(\upsilon^{1, x}_\e > s_{\e}) & \lqq 
\PP(\sup_{t\in [0, s_{\e}]} |\e \xi^{\pi_1}(t)| > \frac{B}{2}\Psi_{0,\e}^\beta s_{\e}) + 1-\PP(\tau_{\pi_1}> s_{\e}) \nonumber\\
& \lqq \exp( - \frac{B}{2}\frac{\Psi_{0, \e}^\beta}{\e^{1-\pi_1}}) + 1-\exp(-\frac{2}{B} \e^{\al \pi_1}  \Psi_{1,_\e}^{\frac{1}{2}(1-\beta)}).
\label{eq: main estimate}
\end{align}
In order to conclude we determine $\pi_1$ and $\Psi_{1, \e}$ such that the last two terms in (\ref{eq: main estimate}) tend to $0$. 
For further use we note that for $\al>1$ we have 
\begin{align*}
\beta^\circ \ln_{\e}(\Psi_{0, \e}) -1 = \frac{\vt \al \beta^\circ (1+\beta^*)- (\vt \al + \beta^* -1 + \beta^*(\al + \beta^\circ -1))}{\vt \al + \beta^* -1 + \beta^*(\al + \beta^\circ -1)}<0, 
\end{align*}
since the denominator is positive by $\vt \al > 1- \beta^\circ$ and 
\begin{align*}
&\vt \al \beta^\circ +\vt \al \beta^\circ \beta^*
- \vt \al - \beta^* +1 - \vt \al\beta^* - \beta^\circ \beta^* +\beta^* \\
&= \vt \al \beta^\circ +\vt \al \beta^\circ \beta^*
- \vt \al  +1 - \vt \al\beta^* - \beta^\circ \beta^* \\
&= \al \vt (\beta^\circ - \beta^*) + \beta^\circ \beta^* (\vt \al -1) - (\vt \al -1)\\
&= \al \vt (\beta^\circ - \beta^*) - (1-\beta^\circ \beta^*) (\vt \al -1) < 0. 
\end{align*}
The right-hand side of (\ref{eq: main estimate}) yields 
\begin{align}
&-\pi_1 > \beta \ln_\e(\Psi_{0, \e})-1\label{eq: aux2}\\
&-\pi_1 < \frac{(1-\beta)}{2\al} \ln_\e(\Psi_{1,_\e}).\label{eq: aux3}
\end{align}
First note that (\ref{eq: aux2}) is satisfied for any $\pi_1<0$ 
since we impose that $\ln_\e(\Psi_{0,\e})>0$, while 
inequality (\ref{eq: aux3}) represents a restriction on $\ln_\e(\Psi_{1, \e})$, 
which can be circumvented for $-\pi_1$ small enough. 

In this case we only have to take into account (\ref{eq: aux2}) and (\ref{eq: aux3}). 
We can hence may choose the desired quantities
\begin{align}
&\Psi_{1,_\e} := 3\delta_\e^\frac{1}{2} \qquad \Big(\vzv \quad \ln_{\e}( \Psi_{1, \e} = \frac{1}{2} \ln_{\e}(3 \delta_\e)~\Big)\nonumber\\
&(-\pi_1) := \frac{(1-\beta)\gamma}{2} \ln_\e(\Psi_{1,_\e}).\label{def: pi1}
\end{align}
\end{proof}

\newpage

\section{The solution selection problem: Proof of Theorem \ref{main theorem}}\label{sec: Hauptsatzbeweis}

In this section we prove a slightly stronger statement than Theorem \ref{main theorem}. 
In the sequel we collect all tailor-made partial result of this article. 
Let $\al, \beta^+, \beta^-$ with $\al > 1- (\beta^+ \wedge \beta^-)$ be given. 

In Corollary \ref{cor: short time scale convergence} the time scale $\ti m_\e$ 
is bounded by $\delta_{\e}^{-\frac{\beta^2}{2}}$. 
By definition (\ref{def: delta eps}) there is $\theta^*>0$ such that 
$\e^{-\theta^*} / \delta_{\e}^{-\frac{\beta^2}{2}} \ra 0$ as $\e \ra 0$. 
Recall $(\Theta^+_{\e}, \Theta^-_{\e}, t_{\e})_{\e\in (0,1]}$ 
defined by Definition~\ref{def: Theta-t} and Lemma~\ref{lem: Hilfslemma 1} 
and the respective hitting times as defined by (\ref{eq: null austritt})
\begin{align*}
\tau_{\Theta^+_\e, \Theta^-_\e}(\e, x) &= \inf\{ t>0~|~ X^{\e, x}_t <-\Theta^-_\e \mbox{ or } X^{\e, x}_t\gqq \Theta^+_\e\}\\
\si_{\delta_\e^+, \delta_\e^-}(\e, x) &= \inf\{ t>0~|~ X^{x, \e}_t< -3\delta_\e^-\mbox{ or }X^{x, \e}_t> 3\delta_\e^+ \},
\end{align*}
where we dropped the dependence on $\vt$. 
Fix a time scale $\hat t_\e = t_\e |\ln(\e)|$ chosen according to Proposition \ref{Prop noise 2} with respect to $t_\e$ 
and $\hat s_\e = s_\e |\ln(\e)|$ with respect to 
$s_\e$ determined by Proposition \ref{prop: linearized dynamics}.
Furthermore we recall the exponents $\kappa<0$ defined in (\ref{def: kappa}) and $\pi<0$ defined in (\ref{def: pi}) 
and $\tau_\kappa$ and $\tau_\pi$ for  
\[
\tau_c = \inf\{t>0~|~|\Delta L_t|>\e^{-c}\}, \qquad c\in \RR. 
\]
Since all other dependencies are clear we shall
write shorthand $\chi = \tau_{\Theta^+_\e, \Theta^-_\e}(\e, 0)$ 
and $\si^x = \si_{\delta_\e^+, \delta_\e^-}(\e, x)$ and $X^x = X^{\e, x}$. 
The solution $X^{\e, x}$ will be denoted by $X^{x}$. 
We use the strong Markov property of $X^{x}$ to control the exit 
from the neighborhood $(-\Theta^-_\e, \Theta^+_\e)$ of the origin. 
For all $\e$ sufficiently small such that $\hat t_\e \lqq \e^{-\theta^*}$ and any $f$ a bounded, uniformly continuous 
function we have 
\begin{align*}
&\EE[f((X^{0}_t)_{t\in [0, \e^{-\theta^*}]})] \\
& = \EE[\EE[f((X^{0}_t)_{t\in [0, \e^{-\theta^*}]})\ind\{\chi \lqq \hat t_\e\} \ind\{\tau_\kappa \gqq \hat t_\e\}
\big(\ind\{X^{0}_{\chi} \gqq \Theta^+_\e\} + \ind\{X^{0}_{\chi} \lqq -\Theta^-_\e\}\big)
~|~\fF_{\chi}]] \\
&\qquad + \PP(\chi > \hat t_\e) + \PP(\tau_\kappa < t_\e)\\
& \lqq \PP(X^{0}_{\chi} \gqq \Theta^+_\e) \sup_{\Theta^+_\e\lqq x \lqq \Theta^+_\e+\e^{1- \kappa}} 
\EE[f((X^{x})_{t\in [0, \e^{-\theta^*}-\chi]})\ind\{\chi \lqq \hat t_\e\}]\\
&\qquad + \PP(X^{0}_{\chi}\lqq -\Theta^-_\e) \sup_{- \Theta^- - \e^{1-\kappa} \lqq x\lqq -\Theta^-_\e} 
\EE[f((X^{x})_{t\in [0, \e^{-\theta^*}-\chi]})\ind\{\chi \lqq \hat t_\e\}] \\
&\qquad + \PP(\chi > \hat t_\e) + \PP(\tau_\kappa < t_\e).
\end{align*}
Proposition \ref{prop: exit location} and Definition \ref{def: Theta-t} 
establish that the probability $\PP(X^{0}_{\chi} \gqq \Theta^+_\e)$ 
tends to $p^+$ as $\e \ra 0$ as given the statement of Theorem \ref{main theorem}. 
The penultimate term tends to $0$ due to Proposition~\ref{Prop noise 2} and the last one due to relation (\ref{eq: tau kappa}). 
In the following we first consider the positive branch. 
Since $\kappa<0$ and we have $\e^{1-\kappa} 3 \delta_\e^+ \ra 0$ as $\e \ra 0$ in addition $\hat s_\e \lqq \e^{-\theta^*}$. 
Hence for $\e$ sufficiently small 
\begin{align*}
&\sup_{\Theta^+_\e\lqq x \lqq \Theta^+_\e+\e^{1- \kappa}} 
\EE[f((X^{x})_{t\in [0, \e^{-\theta^*}-\chi]})\ind\{\chi \lqq \hat t_\e\}]\\
&\lqq \sup_{\Theta^+_\e \lqq x < 3 \delta^+_\e} \EE[f((X^{x})_{t\in [0, \e^{-\theta^*}]})] \\
& \lqq \sup_{\Theta^+_\e \lqq x < 3 \delta^+_\e} \EE[f((X^{\e, x})_{t\in [0, \e^{-\theta^*}]}) 
\ind\{\si^{x} \lqq \hat s_\e\}\ind\{X^{x}_{\si^{x}} \gqq 3\delta_\e\}]\\
&\quad + \sup_{\Theta^+_\e \lqq x < 3 \delta^+_\e} \PP(X^{x}_{\si^{x}} < \Theta^+_\e)
+ \sup_{\Theta^+_\e \lqq x < 3 \delta^+_\e} \PP(\si^{x} > \hat s_\e),
\end{align*}
where the next-to-last and the last term tend to $0$ by Proposition \ref{prop: linearized dynamics} as $\e \ra 0$. 
We continue with the strong Markov property 
\begin{align*}
&\sup_{\Theta^+_\e \lqq x < 3 \delta^+_\e} \EE[f((X^{x})_{t\in [0, \e^{-\theta^*}]}) 
\ind\{\si^{x} \lqq \hat s_\e \}\ind\{\tau_\pi > \hat s_\e\} \ind\{X^{x}_{\si^{\e}} \gqq 3\delta^+_\e\}]
+ \PP(\tau_\pi \lqq \hat s_\e)\\
&\lqq \sup_{3 \delta^+_\e\lqq x \lqq 3 \delta^+_\e+ \e^{1-\pi}} 
\EE[f((X^{x})_{t\in [0, \e^{-\theta^*}-\si^{x}])} \ind\{\si^x \lqq \hat s_\e\}]\\
&\lqq \sup_{3 \delta^+_\e\lqq x \lqq 3 \delta^+_\e+ \e^{1-\pi}} \EE[f((X^{x})_{t\in [0, \e^{-\theta^*}]})].
\end{align*}
Again since $\pi<0$ we have that $\e^{1-\pi} \delta_\e^+ \ra 0$ as $\e \ra 0$. 
First let $f$ be uniformly continuous with respect to $\DD([0,\infty); \RR)$ equipped with the uniform norm. 
We denote by $\Xi$ the module of uniform continuity of $f$. 
For $\Delta_\e = 3 \delta^+_\e+ \e^{1-\pi}$ 
in the statement of Corollary \ref{cor: short time scale convergence} we have 
\begin{align*}
&\sup_{3 \delta^+_\e\lqq x \lqq 3 \delta^+_\e+ \e^{1-\pi}} \EE[f((X^{x})_{t\in [0, \e^{-\theta^*}]})]\\
&\lqq \sup_{3 \delta^+_\e\lqq x \lqq 3 \delta^+_\e+ \e^{1-\pi}} 
\EE[f((X^{x})_{t\in [0, \e^{-\theta^*}]})\{\sup_{t\in [0, \e^{-\theta^*}]} 
|X^{x}_t - x^+_t|\lqq (\delta^+_\e)^{\frac{\beta^+(1-\beta^+)}{2}}\}] \\
&\qquad +\|f\|_\infty \sup_{3 \delta^+_\e\lqq x \lqq 3 \delta^+_\e+ \e^{1-\pi}} 
\PP(\sup_{t\in [0, \e^{-\theta^*}]} |X^{x}_t - x^+_t|>(\delta^+_\e)^{\frac{\beta^+(1-\beta^+)}{2}})\\
&\lqq f((x^+_t)_{t\in [0, \e^{-\theta^*}]}) + \Xi(\delta_\e^{\frac{\beta^+(1-\beta^+)}{2}})
+ \|f\|_\infty \sup_{3 \delta^+_\e\lqq x \lqq 3 \delta^+_\e+ \e^{1-\pi}} 
\PP(\sup_{t\in [0, \e^{-\theta^*}]} |X^{x}_t - x^+_t|>(\delta^+_\e)^{\frac{\beta^+(1-\beta^+)}{2}}).
\end{align*}
Corollary \ref{cor: short time scale convergence} yields that the last term converges to $0$. 
The negative branch is treated analogously. 
For the case of general case of $f$ not uniformly continuous, we define the cutoff function 
$f_m(x) := f(x) \ind\{-m \lqq x\lqq m\}$, 
which is uniformly continuous 
and finally send $m$ to infinity, which is justified by the Beppo-Levi theorem. 

We prove the lower bound. Let $f$ be uniformly continuous. 
\begin{align*}
&\EE[f((X^{0}_t)_{t\in [0, \e^{-\theta^*}]})] \\
&\gqq \EE[f((X^{0}_t)_{t\in [0, \e^{-\theta^*}]})(\ind\{\chi \lqq \hat t_\e\} 
+ \ind\{\chi > \hat t_\e\})] \\
& = \EE[\EE[f((X^{\e, 0}_t)_{t\in [0, \e^{-\theta^*}]})\ind\{\chi \lqq \hat t_\e\} 
\big(\ind\{X^{0}_{\chi} \gqq \Theta^+_\e\} + \ind\{X^{0}_{\chi} \lqq -\Theta^-_\e\}\big)
~|~\fF_{\chi}]] \\
& \gqq \PP(X^{0}_{\chi} \gqq \Theta^+_\e) \sup_{ x\lqq \Theta^+_\e } 
\EE[f((X^{\e, x})_{t\in [0, \e^{-\theta^*}-\chi]})] \\
&\qquad + \PP(X^{0}_{\chi} \gqq -\Theta^-_\e) \sup_{ x\lqq -\Theta^-_\e} 
\EE[f((X^{x})_{t\in [0, \e^{-\theta^*}-\chi]})]
\end{align*}
Application of Proposition \ref{prop: exit location} and Definition \ref{def: Theta-t} 
ensure again that $\lim_{\e \ra 0} \PP(X^{0}_{\chi} \gqq \Theta^+_\e) = p^+$. 
We may continue with the positive branch
\begin{align*}
&\sup_{x\gqq \Theta^+_\e} \EE[f((X^{x})_{t\in [0, \e^{-\theta^*}-\hat t_\e]})] \\
&\gqq \sup_{x\gqq \Theta^+_\e} \EE[f((X^{\e, x})_{t\in [0, \e^{-\theta^*}-\hat t_\e]})
\ind\{X^{x}_{\si^{x}} \gqq 3\delta^+_\e\} \ind\{\si^{x} \lqq s_\e\}] \\
&\gqq \sup_{x\gqq 3\delta^+_\e} \EE[f((X^{x})_{t\in [0, \e^{-\theta^*}-\hat t_\e- s_\e]})
\ind\{\sup_{t\in [0, \e^{-\theta^*}]} |X^{x}_t - x^+_t|\lqq \delta_\e^{\frac{\beta^+(1-\beta^+)}{2}}\}] \\
&\gqq f((x^+_t)_{t\in [0, \e^{-\theta^*}-\hat t_\e- s_\e]}) - \Xi((\delta^+_\e)^{\frac{\beta^+(1-\beta^+)}{2}}). 
\end{align*}
The negative branch is treated analogously. For a function $f$ not uniformly continuous 
we use the same truncation argument as before. This proves the desired result.

\section*{Acknowledgements}
The second author would like to thank the Cooperation Group 
``Exploring climate variability: physical models, 
statistical inference and stochastic dynamics'' (February 18 --- March 28, 2013) 
organized by Peter Imkeller, Ilya Pavlyukevich and Holger Kantz, 
which was cordially hosted at ZiF Bielefeld, where this work was begun. 
He would further like to thank the Potsdam probability group, namely 
Mathias Rafler for helpful discussions and Sylvie Roelly for her constant support.

\end{document}